\documentclass[11pt,a4paper]{article}

\usepackage[utf8]{inputenc}
\usepackage{amsmath}
\usepackage{amsfonts}
\usepackage{amssymb}
\usepackage{amsthm}
\usepackage{bbm}
\usepackage{xcolor}
\usepackage{graphicx}
\usepackage{hyperref}
\usepackage[top=2cm,
  bottom=2cm,
  left=2cm,
  right=2cm]{geometry}
\numberwithin{equation}{section}

\newtheorem{theorem}{Theorem}[section]
\newtheorem{lemma}[theorem]{Lemma}

\newtheorem{proposition}[theorem]{Proposition}
\newtheorem{corollary}{Corollary}[section]

\theoremstyle{definition}
\newtheorem{remark}[theorem]{Remark}
\newtheorem{example}[theorem]{Example}

\title{Sandwiched SDEs with unbounded drift driven by H\"older noises}

\author{Giulia Di Nunno$^{1,2}$\\\href{mailto:giulian@math.uio.no}{giulian@math.uio.no} 
   \and Yuliya Mishura$^3$ \\ \href{mailto:myus@univ.kiev.ua}{myus@univ.kiev.ua} 
   \and Anton Yurchenko-Tytarenko$^1$ \\ \href{mailto:antony@math.uio.no}{antony@math.uio.no}}

\date{%
    $^1$Department of Mathematics, University of Oslo\\%
    $^2$Department of Business and Management Science, NHH Norwegian School of Economics, Bergen\\
    $^3$Department of Probability, Statistics and Actuarial Mathematics, Taras Shevchenko National University of Kyiv\\[2ex]%
    \today
}

\begin{document}

\maketitle
\abstract{We study a stochastic differential equation with an unbounded drift and general H\"older continuous noise of an arbitrary order. The corresponding equation turns out to have a unique solution that, depending on a particular shape of the drift, either stays above some continuous function or has continuous upper and lower bounds. Under some additional assumptions on the noise, we prove that the solution has moments of all orders. We complete the study providing a numerical scheme for the solution. As an illustration of our results and motivation for applications, we suggest two stochastic volatility models which we regard as generalizations of the CIR and CEV processes.
\\[11pt]
\textbf{Keywords:} sandwiched process, unbounded drift, H\"older continuous noise, numerical scheme, stochastic volatility\\
\textbf{MSC 2020:} 60H10; 60H35; 60G22; 91G30}

\section*{Introduction}

Stochastic differential equations (SDEs) the solutions of which take values in a given bounded domain are largely applied in several fields. Just as an illustration, we can consider the Tsallis--Stariolo--Borland (TSB) model employed in biophysics defined as
\begin{equation}\label{eq: TSB}
    dY_1(t) = - \frac{\theta Y_1(t)}{1-Y^2_1(t)}dt + \sigma dW(t), \quad \theta>0,~\sigma>0, 
\end{equation}
with $W$ being a standard Wiener process. If $\frac{\sigma^2}{\theta} \in (0,1]$, the TSB process is ``sandwiched'' between $-1$ and $1$ (for more details, see e.g. \cite[Subsection 2.3]{Domingo2019} or \cite[Chapter 3 and Chapter 8]{BoundedNoises2013}). Another example is the Cox--Ingersoll--Ross (CIR) process \cite{COX1981, COX1985-1, COX1985-2} defined via the SDE of the form
\begin{equation*}
    dX(t) = (\theta_1 - \theta_2 X(t))dt + \sigma \sqrt{X(t)}dW(t), \quad \theta_1,~\theta_2,~\sigma >0.
\end{equation*}
Under the so-called Feller condition $2\theta_1 > \sigma^2$, the CIR process is lower bounded (more precisely, positive) a.s. which justifies its popularity in interest rates and stochastic volatility modeling in finance. Moreover, by \cite[Theorem 2.3]{MYT2021}, the square root $Y_2(t) := \sqrt{X(t)}$ of the CIR process satisfies the SDE of the form
\begin{equation}\label{eq: standard CIR square root}
    dY_2(t) = \frac{1}{2}\left(\frac{\theta_1 - \sigma^2/4}{Y_2(t)} - \theta_2Y_2(t)\right)dt + \frac{\sigma}{2}dW(t)
\end{equation}
and SDEs \eqref{eq: TSB} and \eqref{eq: standard CIR square root} both have an additive noise term and an unbounded drift with points of singularity at the bounds ($\pm 1$ for the TSB and $0$ for CIR) which have a ``repelling'' action so that the corresponding processes never cross or even touch the bounds.

The goal of this paper is to study a family of SDEs of the type similar to \eqref{eq: TSB} and \eqref{eq: standard CIR square root}, namely
\begin{equation}\label{volatility1}
    Y_t = Y_0 + \int_0^t b(s, Y_s) ds + Z_t, \quad t\in[0,T],
\end{equation}
where where the drift $b$ is unbounded. We consider separately two cases.
\begin{itemize}
    \item[\textbf{(A)}] $b$ is a real function defined on the set $\{(t,y)\in[0,T]\times\mathbb R~|~y > \varphi(t)\}$ such that $b(t,y)$ has an explosive growth of the type $(y - \varphi(t))^{-\gamma}$ as $y \downarrow \varphi(t)$, where $\varphi$ is a given H\"older continuous function and $\gamma > 0$. The we will see that the process $Y$ satisfying \eqref{volatility1} is lower bounded by $\varphi$, i.e.
    \begin{equation}\label{eq: Y is one-sandwiched introduction}
        Y_t > \varphi(t), \quad a.s., \quad t\in[0,T],
    \end{equation} 
    which can be called a one-sided sandwich.
    \item[\textbf{(B)}] $b$ is a real function defined on the set $\{(t,y)\in[0,T]\times\mathbb R~|~\varphi(t) < y < \psi(t)\}$ such that $b(t,y)$ has an explosive growth of the type $(y - \varphi(t))^{-\gamma}$ as $y \downarrow \varphi(t)$ and an explosive decrease of the type $-(\psi(t) - y)^{-\gamma}$ as $y \uparrow \psi(t)$, where $\varphi$ and $\psi$ are given H\"older continuous functions such that $\varphi(t) < \psi(t)$, $t\in[0,T]$, and $\gamma > 0$. We will see that in this case the solution to \eqref{volatility1} turns out to be sandwiched, namely
    \begin{equation}\label{eq: Y is sanwiched introduction}
        \varphi(t) < Y_t < \psi(t) \quad a.s., \quad t\in[0,T],
    \end{equation}
    as a two-sided sandwich.
\end{itemize}

The noise term $Z$ in \eqref{volatility1} is an arbitrary $\lambda$-H\"older continuous noise, $\lambda\in(0,1)$. Our main motivation to consider $Z$ from such a general class instead of the classical Wiener process lies in the desire to go beyond Markovianity and include memory in the dynamics \eqref{volatility1} via the noise term. It should be noted that the presence of memory is a commonly observed empirical phenomenon (in this regard, we refer the reader to \cite[Chapter 1]{Beran}, where examples of datasets with long memory are collected, and to \cite{Samor} for more details on stochastic processes with long memory) and the particular application which we have in mind throughout this paper comes from finance where the presence of market memory is well-known and extensively studied (see e.g. \cite{Anh2005, Ding1993, Yamasaki2005} or \cite{Tarasov2019} for a detailed historical overview on the subject). In the current literature, processes with memory in the noise have been used as stochastic volatilities allowing for the inclusion of empirically detected features, such as volatility smiles and skews in long term options \cite{Comte2010}, see also \cite{Bollerslev1996, Chronopoulou2010} for more details on long memory models and \cite{Gatheral2018} for short memory coming from the microstructure of the market. Some studies (see e.g. \cite{Alfi_Coccetti_Petri_Pietronero_2007}) indicate that the roughness of the volatility changes over time which justifies the choice of multifractional Brownian motion \cite{Ayache_Peng_2012} or even general Gaussian Volterra processes \cite{Merino_Pospisil_Sobotka_Sottinen_Vives_2021} as drivers. Separately we mention the series of papers \cite{MYuT2018, MYuT2019, Mishura2020} that study an SDE of the type \eqref{eq: standard CIR square root} with memory introduced via a fractional Brownian motion with $H> \frac{1}{2}$: 
\begin{equation}\label{eq: CIR eq introduction}
dY_t = \left(\frac{\theta_1}{Y_t} - \theta_2 Y_t\right) dt + \sigma dB^H_t, \quad \theta_1,~\theta_2,~\sigma >0,~t\in[0,T].
\end{equation}
Our model \eqref{volatility1} can thus be regarded as a generalization of \eqref{eq: CIR eq introduction} accommodating a much flexible choice of noise to deal with problems of local roughness mentioned above.

In this paper, we first consider existence and uniqueness of solution to \eqref{volatility1} and then focus on the moments of both positive and negative orders. We recognize that similar problems concerning equations of the type \eqref{volatility1} with lower bound $\varphi \equiv 0$ and the noise $Z$ being a fractional Brownian motion with $H>\frac{1}{2}$ were addressed in \cite{Hu2007}. There, authors used pathwise arguments to prove existence and uniqueness of the solution whereas a Malliavin calculus-based method was applied to obtain finiteness of the inverse moments. Despite its elegance, the latter technique requires the noise to be Gaussian and, moreover, is unable to ensure the finiteness of the inverse moments on the entire time interval $[0,T]$. It should be stressed that the inverse moments are crucial for numerical simulation of \eqref{volatility1} since it is necessary to control the explosive growth of the drift near bounds, and hence any limitations embedded in the techniques used to prove results on the inverse moments directly influence the convergence of the numerical schemes. These disadvantages of the Malliavin method mentioned above resulted in restrictive conditions involving all parameters of the model and $T$ in numerical schemes in e.g. \cite[Theorem 4.2]{Hong2020} and \cite[Theorem 4.1]{ZhYu2020}.

The approach we have is to use pathwise calculus together with stopping times arguments also for the inverse moments. This allows us, on the one hand, to choose from a much broader family of noises well beyond the Gaussian framework and, on the other hand, to prove the existence of the inverse moments of the solution on the entire $[0,T]$. The corresponding inverse moment bounds are presented in Theorems \ref{lower bound for volatility} and \ref{th: moments of sandwiched Y} and can be regarded as the main results of the paper.

This paper is organised as follows. In Section \ref{sec: prel and assum}, the general framework is described and the main assumptions are listed. Furthermore, some examples of possible noises $Z$ are provided (including continuous martingales and Gaussian Volterra processes). In Section \ref{sec: one-sided bound}, we provide existence and uniqueness of the solution to \eqref{volatility1} as well as discuss property \eqref{eq: Y is one-sandwiched introduction}, derive upper and lower bounds for the solution in terms of the noise and study finiteness of $\mathbb E\left[ \sup_{t\in[0,T]} |Y_t|^r\right]$ and $\mathbb E\left[ \sup_{t\in[0,T]} (Y_t - \varphi(t))^{-r}\right]$, $r\ge 1$. Full details of the proof of the existence are provided in the Appendix \ref{sec: existence of local solution}. Section \ref{sec: two-sided bounds} is devoted to the sandwiched case \eqref{eq: Y is sanwiched introduction}: existence, uniqueness and properties of the solution are discussed. Our approach is readily put in application introducing the generalized CIR and CEV (see \cite{Andersen2006, Cox1975NotesOO}) processes in Section \ref{sec: examples}. Finally, in Section \ref{sec: approximation schemes} we provide a modification of the standard Euler numerical scheme which we call \emph{semi-heuristic} in view of the type of convergence and dependence on a random variable that cannot be computed explicitly and has to be estimated from the discretised data. The results are illustrated by simulations.

\section{Preliminaries and assumptions}\label{sec: prel and assum}

In this section, we present the framework and collect all the assumptions regarding both the noise $Z$ and the drift functional $b$ from equation \eqref{volatility1}. 

\subsection{The noise}\label{ss: assumptions on noise}

Fix $T>0$ and consider two $\lambda$-H\"older continuous functions $\varphi$, $\psi$: $[0,T]\to\mathbb R$, $\lambda\in(0,1)$, such that $\varphi(t) < \psi(t)$ for all $t\in[0,T]$. Throughout this paper, the noise term $Z = \{Z_t,~t\in[0,T]\}$ in equation \eqref{volatility1} is an arbitrary stochastic process such that:
\begin{itemize}
\item[\textbf{(Z1)}] $Z_0 = 0$ a.s.;
\item[\textbf{(Z2)}] $Z$ has H\"older continuous paths of the same order $\lambda$ as functions $\varphi$ and $\psi$, i.e. there exists a random variable $\Lambda = \Lambda_{\lambda}(\omega)\in(0,\infty)$ such that
\begin{equation}\label{eq: Holder continuity of Z}
|Z_t - Z_s| \le \Lambda |t-s|^{\lambda}, \quad t,s \in [0,T].
\end{equation}
\end{itemize}

Note that we do not require any particular assumptions on distribution of the noise (e.g. Gaussianity), but, for some results, we will need the random variable $\Lambda$ from \eqref{eq: Holder continuity of Z} to have moments of sufficiently high orders. In what follows, we list several examples of admissible noises as well as properties of the corresponding random variable $\Lambda$. In order to discuss the latter, we will use a corollary from the well-known Garsia-Rodemich-Rumsey inequality (see \cite{GRR1971} for more details).

\begin{lemma}\label{l: GRR}
    Let $f$: $[0,T]\to\mathbb R$ be a continuous function, $p \ge 1$ and $\alpha > \frac{1}{p}$. Then for all $t,s \in [0,T]$ one has
    \[
        |f(t) - f(s)| \le A_{\alpha,p} |t-s|^{\alpha - \frac{1}{p}} \left(\int_0^T \int_0^T \frac{|f(x) - f(y)|^p}{|x-y|^{\alpha p + 1}} dx dy\right)^{\frac{1}{p}},
    \]
    with the convention $0/0 = 0$, where 
    \begin{equation}\label{eq: GRR constant}
        A_{\alpha, p} = 2^{3 + \frac{2}{p}}\left( \frac{\alpha p + 1}{\alpha p - 1} \right).
    \end{equation}
\end{lemma}
\noindent Note that this lemma was stated, for example, in \cite{Nualart2002} and \cite{ASVY2014} without computing the constant $A_{\alpha, p}$ explicitly, but we will need the latter for the approximation scheme in section \ref{sec: approximation schemes}.

\begin{proof}
    The proof can be easily obtained from \cite[Lemma 1.1]{GRR1971} by putting in the notation of \cite{GRR1971} $\Psi(u):= |u|^\beta$ and $p(u):= |u|^{\alpha+\frac{1}{\beta}}$, where $\beta = p \ge 1$ in our statement.   
\end{proof}


\begin{example}[\emph{H\"older continuous Gaussian processes}]\label{ex: Gaussian Holder processes}
    Let $Z = \{Z_t,~t\in 0\}$ be a centered Gaussian process with $Z_0 = 0$ and $H\in(0,1)$ be a given constant. Then, by \cite{ASVY2014}, $Z$ has a modification with H\"older continuous paths of any order $\lambda\in(0,H)$ if and only if for any $\lambda\in(0,H)$ there exists a constant $C_{\lambda} > 0$ such that
    \begin{equation}\label{eq: condition on second moment that ensures Holder}
        \left( \mathbb E |Z_t - Z_s|^{2} \right)^{\frac{1}{2}} \le C_{\lambda}|t-s|^{\lambda}, \quad s,t\in[0,T].
    \end{equation}
    Furthermore, according to \cite[Corollary 3]{ASVY2014}, the class of all Gaussian processes on $[0,T]$, $T\in(0,\infty)$, with H\"older modifications of any order $\lambda\in(0,H)$ consists exclusively of Gaussian Fredholm processes
    \[
        Z_t = \int_0^T \mathcal K(t,s) dB_s, \quad t\in[0,T],
    \]
    with $B = \{B_t,~t\in[0,T]\}$ being some Brownian motion and $\mathcal K \in L^2([0,T]^2)$ satisfying, for all $\lambda\in(0,H)$,
    \[
        \int_0^T |\mathcal K(t,u) - \mathcal K(s,u)|^2 du \le C_\lambda |t-s|^{2\lambda}, \quad s,t\in [0,T],
    \]
    where $C_\lambda > 0$ is some constant depending on $\lambda$.
    
    Finally, using Lemma \ref{l: GRR}, one can prove that the corresponding random variable $\Lambda$ can be chosen to have moments of all positive orders. Namely, assume that $\lambda \in (0,H)$ and take $p\ge 1$ such that $\frac{1}{p} < H - \lambda$. If we take
    \begin{equation}\label{eq: explicit form of Lambda}
        \Lambda =  A_{\lambda + \frac{1}{p}, p} \left(\int_0^T \int_0^T \frac{|Z(x) - Z(y)|^p}{|x-y|^{\lambda p + 2}} dx dy\right)^{\frac{1}{p}},
    \end{equation}
    then, for any $r\ge 1$
    \[
        \mathbb E \Lambda^r < \infty
    \]
    and for all $s,t\in[0,T]$:
    \[
        |Z_t - Z_s| \le \Lambda |t-s|^\lambda,
    \]
    see e.g. \cite[Lemma 7.4]{Nualart2002} for fractional Brownian motion or \cite[Theorem 1]{ASVY2014} for the general Gaussian case.
\end{example}

In particular, the condition \eqref{eq: condition on second moment that ensures Holder} presented in Example \ref{ex: Gaussian Holder processes} is satisfied by the following stochastic processes.

\begin{example}[\emph{fractional Brownian motion}]
    Fractional Brownian motion $B^H = \{B^H_t,~t\ge0\}$ with $H\in(0,1)$ (see e.g. \cite{Nourdin2012}) since 
    \[
        \left(\mathbb E |B^H_t - B^H_s|^2 \right)^{\frac{1}{2}} = |t - s|^{H} \le T^{H-\lambda} |t-s|^{\lambda},
    \]
    i.e. $B^H$ has a modification with H\"older continuous paths of any order $\lambda\in(0,H)$.
\end{example}

\begin{example}[\emph{Gaussian Volterra processes with fBm-type kernel}]\label{ex: Gaussian Volterra processes}
    Gaussian Volterra processes
    \[
        Z_t = \int_0^t \mathcal K (t,s) dB_s, \quad t\in [0,T],
    \] 
    with the kernel of the form
    \[
        \mathcal K(t,s) = a(s) \int_s^t b(u) c(u-s) du \mathbbm 1_{s<t},
    \]
    where $a\in L^p [0,T]$, $b \in L^q[0,T]$ and $c\in L^r[0,T]$ with $p$, $q$, $r$ such that
    \begin{itemize}
        \item[1)] $p\in [2,\infty]$, $q\in (1,\infty]$, $r\in [1,\infty]$,
        \item[2)] $\frac{1}{p} + \frac{1}{r} \ge \frac{1}{2}$,
        \item[3)] $\frac{1}{p} + \frac{1}{q} + \frac{1}{r} < \frac{3}{2}$.
    \end{itemize}
    Under the conditions specified above, the process $Z$ satisfies (see \cite[Lemma 1]{MSS})
    \[
        \left( \mathbb E |Z_t - Z_s|^{2} \right)^{\frac{1}{2}} \le \lVert a \rVert_p \lVert b \rVert_q \lVert c \rVert_r |t-s|^{ \frac{3}{2} - \frac{1}{p} - \frac{1}{q} - \frac{1}{r} }, \quad t,s\in[0,T],
    \]
    and therefore has a modification with H\"older continuous paths of all orders $\lambda\in\left(0, \frac{3}{2} - \frac{1}{p} - \frac{1}{q} - \frac{1}{r}\right)$.
\end{example}

\begin{example}[\emph{multifractional Brownian motion}]\label{ex: mBm}
    The harmonizable multifractional Brownian motion $Z = \{Z_t,~t\in[0,T]\}$ with functional parameter $H$: $[0,T]\to (0,1)$ (for more detail on this process, see e.g. \cite{Benassi1997}, \cite{ST2006},  \cite{DKMR} and references therein). Namely,
    \[
        Z_t := \int_{\mathbb R} \frac{e^{itu} - 1}{|u|^{H_t + \frac{1}{2}}} \widetilde{W}(du), \quad t\in [0,T],
    \]
    where $\widetilde{W}(du)$ is a unique Gaussian complex-valued random measure such that for all $f\in L^2(\mathbb R)$ 
    \[
        \int_{\mathbb R} f(u) W(du) = \int_{\mathbb R} \widehat f(u) \widetilde{W}(du) \quad \text{a.s.}
    \]
    Also let $H$ satisfy the following assumptions:
    \begin{itemize}
        \item[1)] there exist constants $0 < h_1 < h_2 < 1$ such that for any $t\in [0,T]$
        \[
            h_1 < H_t < h_2,
        \]
        \item[2)] there exist constants $D>0$ and $\alpha\in(0,1]$ such that
        \[
            |H_t - H_s| \le D |t-s|^{\alpha}, \quad t,s\in [0,T].
        \]
    \end{itemize}
    \noindent Then, according to Lemma 3.1 from \cite{DKMR}, there is a constant $C > 0$ such that for all $s,t\in[0,T]$:
    \[
        \left( \mathbb  E(Z_t - Z_s)^2\right)^{\frac{1}{2}} \le C|t-s|^{h_1 \wedge \alpha}
    \] 
    and, since $Z$ is clearly Gaussian, it has a H\"older continuous modification of any order $\lambda \in (0, h_1 \wedge \alpha)$.
\end{example}

\begin{example}[\emph{non-Gaussian continuous martingales}]\label{ex: non-Gaussian continuous martingales}
    Denote $B = \{B_t,~t\in[0,T]\}$ a standard Brownian motion and $\sigma = \{\sigma_t,~t\in[0,T]\}$ an It\^o integrable process such that, for all $\beta > 0$,
    \begin{equation}\label{eq: condition to make martingale Holder}
        \sup_{u\in[0,T]} \mathbb E \sigma^{2 + 2\beta}_u < \infty.
    \end{equation}
    Define 
    \[
        Z_t := \int_0^t \sigma_u dB_u, \quad t\in [0,T].
    \]
    Then, by the Burkholder-Davis-Gundy inequality, for any $0 \le s<t \le T$ and any $\beta > 0$:
    \begin{equation}\label{eq: condition required for Euler}
    \begin{aligned}
        \mathbb E |Z_t - Z_s|^{2+2\beta} &\le C_\beta \mathbb E\left[\left(\int_s^t \sigma^2_u du\right)^{1+\beta}\right] \le C_\beta (t-s)^{\beta} \int_s^t \mathbb E \sigma^{2 + 2\beta}_u du
        \\
        &\le C_\beta \sup_{u\in[0,T]} \mathbb E \sigma^{2 + 2\beta}_u (t-s)^{1 + \beta}.
    \end{aligned}
    \end{equation}
    Therefore, by the Kolmogorov continuity theorem and an arbitrary choice of $\beta$, $Z$ has a modification that is $\lambda$-H\"older continuous of any order $\lambda \in \left(0, \frac{1}{2}\right)$.
    
    Next, for an arbitrary $\lambda \in \left(0, \frac{1}{2}\right)$, choose $p\ge 1$ such that $\lambda + \frac{1}{p} < \frac{1}{2}$ and put
    \[
        \Lambda :=  A_{\lambda + \frac{1}{p}, p} \left(\int_0^T \int_0^T \frac{|Z(x) - Z(y)|^p}{|x-y|^{\lambda p + 2}} dx dy\right)^{\frac{1}{p}},
    \]
    where $A_{\lambda + \frac{1}{p}, p}$ is defined by \eqref{eq: GRR constant}. By the the Burkholder-Davis-Gundy inequality, for any $r>p$, we obtain
    \[
        \mathbb E |Z_t - Z_s|^{r} \le |t-s|^{\frac{r}{2}} C_r \sup_{u\in[0,T]} \mathbb E \sigma^{r}_u , \quad s,t\in[0,T].
    \]
    Hence, using Lemma \eqref{l: GRR} and the Minkowski integral inequality, we have:
    \begin{align*}
        \left( \mathbb E \Lambda^r \right)^{\frac{p}{r}} &= A^p_{\lambda + \frac{1}{p}, p} \left(\mathbb E\left[ \left( \int_0^T \int_0^T \frac{|Z_u - Z_v|^p}{|u-v|^{\lambda p + 2}} du dv \right)^{\frac{r}{p}}\right]\right)^{\frac{p}{r}}
        \\
        &\le A^p_{\lambda + \frac{1}{p}, p}\int_0^T \int_0^T \frac{\left(\mathbb E [|Z_u - Z_v|^r]\right)^{ \frac{p}{r} }}{|u-v|^{\lambda p + 2}} du dv
        \\
        &\le A^p_{\lambda + \frac{1}{p}, p} C_r^{\frac p r} \left(\sup_{t\in[0,T]} \mathbb E \sigma^{r}_t \right)^{\frac p r} \int_0^T \int_0^T |u-v|^{\frac{p}{2}-\lambda p - 2} du dv < \infty,
    \end{align*}
    since $\frac{p}{2}-\lambda p - 2 > -1$, i.e. $\mathbb E\Lambda^r < \infty$ for all $r>0$. Note that condition \eqref{eq: condition to make martingale Holder} can actually be relaxed (see e.g. \cite[Lemma 14.2]{BMS2018}).
\end{example}

\subsection{The drift}

Set $T\in(0,\infty)$. Let and $\varphi$, $\psi$: $[0,T] \to \mathbb R$, $\varphi(t) < \psi(t)$, $t \in [0,T]$, be $\lambda$-H\"older continuous functions, with $\lambda\in(0,1)$ being the same as in assumption \textbf{(Z2)}, i.e. there exists a constant $K = K_{\lambda}$ such that
\[
    |\varphi(t) - \varphi(s)| + |\psi(t) - \psi(s)| \le K |t-s|^{\lambda}, \quad t,s\in[0,T].
\]
For an arbitrary pair $a_1$, $a_2 \in [-\infty,\infty]$ denote
\begin{equation}\label{eq: def of D}
    \mathcal D_{a_1, a_2} := \{(t,y)~|~t\in[0,T], y\in (\varphi(t)+a_1, \psi(t) - a_2)\}.
\end{equation}
By $\mathcal D_{a_1}$ we shall mean the set $\mathcal D_{a_1, -\infty} = \{(t,y)~|~t\in[0,T], y\in (\varphi(t)+ a_1, \infty)\}$.

Consider the stochastic differential equation of the form \eqref{volatility1}, where $Y(0) > \varphi(0)$ and the drift $b$ is a function satisfying the following assumptions:

\begin{itemize}
    \item[\textbf{(A1)}] $b$: $\mathcal D_{0} \to \mathbb R$ is continuous;
    \item[\textbf{(A2)}] for any $\varepsilon>0$ there is a constant $c_{\varepsilon} > 0$ such that for any $(t,y_1), (t, y_2) \in \mathcal D_{\varepsilon}$:
    \[
        |b(t,y_1) - b(t, y_2)| \le c_{\varepsilon} |y_1 - y_2|;
    \]
    \item[\textbf{(A3)}] there are such positive constants $y_*$, $c$ and $\gamma$ that for all $(t,y) \in \mathcal D_0 \setminus \mathcal D_{y_*}$: 
    \[
        b(t,y) \ge \frac{c}{\left(y- \varphi(t)\right)^\gamma}.
    \]
    \item[\textbf{(A4)}] the constant $\gamma$ from assumption \textbf{(A3)} satisfies condition
    \[
        \gamma > \frac{1-\lambda}{\lambda}
    \]
    with $\lambda$ being the order of H\"older continuity of $\varphi$, $\psi$ and paths of $Z$.
\end{itemize}

\begin{example}
    Let $\alpha_1$: $[0,T] \to (0,\infty)$ be an arbitrary continuous function and $\alpha_2$: $\mathcal D_0 \to \mathbb R$ be such that
    \[
        |\alpha_2(t,y_1) - \alpha_2(t,y_2)| \le C|y_1-y_2|, \quad (t,y_1),(t,y_2)\in \mathcal D_0,
    \]
    for some constant $C>0$. Then
    \[
        b(t,y) := \frac{\alpha_1(t)}{(y - \varphi(t))^\gamma} - \alpha_2(t, y), \quad (t,y) \in \mathcal D_0,
    \]
    satisfies assumptions \textbf{(A1)}--\textbf{(A4)} (provided that $\gamma > \frac{1 - \lambda}{\lambda}$).
\end{example}

In section \ref{sec: two-sided bounds}, we discuss the equation \eqref{volatility1} in a slightly different setting. Namely, we consider $\varphi(0) < Y_0 <\psi(0)$ as well as an alternative list of assumptions on $b$:
\begin{itemize}
    \item[\textbf{(B1)}] $b$: $\mathcal D_{0,0} \to \mathbb R$ is continuous;
    \item[\textbf{(B2)}] for any pair $\varepsilon_1$, $\varepsilon_2 >0$ such that $\varepsilon_1+\varepsilon_2 < \lVert \varphi - \psi\rVert_\infty$ there is a constant $c_{\varepsilon_1, \varepsilon_2} > 0$ such that for any $(t,y_1), (t, y_2) \in \mathcal D_{\varepsilon_1, \varepsilon_2}$:
    \[
        |b(t,y_1) - b(t, y_2)| \le c_{\varepsilon_1, \varepsilon_2} |y_1 - y_2|;
    \]
    \item[\textbf{(B3)}] there are constants $\gamma$, $y_{*} > 0$, $y_{*} < \frac{1}{2}\lVert \varphi - \psi\rVert_\infty$ and $c > 0$ such that for all $(t,y) \in {\mathcal D}_{0, 0} \setminus {\mathcal D}_{y_*, 0}$: 
    \[
        b(t,y) \ge \frac{c}{\left(y- \varphi(t)\right)^\gamma},
    \]
    and for all $(t,y) \in {\mathcal D}_{0, 0} \setminus {\mathcal D}_{0, y_*}$:
    \[
        b(t,y) \le - \frac{c}{\left(\psi(t) - y\right)^\gamma}.
    \]
    \item[\textbf{(B4)}] the constant $\gamma$ from assumption \textbf{(B3)} satisfies condition
    \[
        \gamma > \frac{1-\lambda}{\lambda}
    \]
    with $\lambda$ being the order of H\"older continuity of $\varphi$, $\psi$ and paths of $Z$.
\end{itemize}

\begin{remark}
    We always assume that the initial value $Y_0$ is a deterministic constant such that $Y_0>\varphi(0)$, if the drift $b$ satisfies Assumptions \textbf{(A1)}--\textbf{(A4)}, and $\varphi(0) < Y_0 < \psi(0)$, if the drift $b$ satisfies Assumptions \textbf{(B1)}--\textbf{(B4)}.
\end{remark}

\begin{example}
    Let $\alpha_1$: $[0,T] \to (0,\infty)$, $\alpha_2$: $[0,T] \to (0,\infty)$ and $\alpha_3$: $\mathcal D_{0,0} \to \mathbb R$ be continuous and
    \[
        |\alpha_3(t,y_1) - \alpha_3(t,y_2)| \le C|y_1-y_2|, \quad (t,y_1),(t,y_2)\in \mathcal D_{0,0},
    \]
    for some constant $C>0$.
    Then
    \[
        b(t,y) := \frac{\alpha_1(t)}{(y - \varphi(t))^\gamma}-  \frac{\alpha_2(t)}{(\psi(t) - y)^\gamma} - \alpha_3(t,y), \quad t\in[0,T], y\in\mathcal D_{0,0},
    \]
    satisfies assumptions \textbf{(B1)}--\textbf{(B4)} provided that $\gamma > \frac{1 - \lambda}{\lambda}$.
\end{example}

\section{SDE with lower-sandwiched solution case}\label{sec: one-sided bound}

In this section, we discuss existence, uniqueness and properties of the solution of \eqref{volatility1} under assumptions \textbf{(A1)}--\textbf{(A4)}. First, we demonstrate that \textbf{(A1)}--\textbf{(A3)} ensure the existence and uniqueness of the solution to \eqref{volatility1} until the first moment of hitting the lower bound $\{\varphi(t),~t\in[0,T]\}$ and then we prove that \textbf{(A4)} guarantees that the solution exists on the entire $[0,T]$, since it always stays above $\varphi(t)$. The latter property justifies the name \emph{lower-sandwiched} in the section title.

Finally, we derive additional properties of the solution, still in terms of some form of bounds.

\begin{remark}
    Throughout this paper, the pathwise approach will be used, i.e. we fix a H\"older continuous trajectory of $Z$ in most proofs. For simplicity, we omit $\omega$ in brackets in what follows.
\end{remark}

\subsection{Existence and uniqueness result}

As mentioned before, we shall start from the existence and uniqueness of the local solution.

\begin{theorem}\label{thm: existence of the local solution}
    Let assumptions \textbf{(A1)}--\textbf{(A3)} hold. Then SDE \eqref{volatility1} has a unique local solution in the following sense: there exists a continuous process $Y = \{Y_t,~t\in [0,T]\}$ such that
    \[
        Y_t = Y_0 + \int_0^t b(s, Y_s) ds + Z_t, \quad \forall t\in[0,\tau_0],
    \]
    with 
    \begin{align*}
        \tau_0 :&= \sup\{t\in[0,T]~|~\forall s \in [0,t): Y_s > \varphi(s)\}
        \\
        &=\inf\{t\in[0,T]~|~Y_t = \varphi(t)\} \wedge T.
    \end{align*}

    Furthermore, if $\tilde Y$ is another process satisfying equation \eqref{volatility1} on any interval $[0,t] \subset [0,\tilde\tau_0)$, where
    \[
        \tilde\tau_0 := \sup\{s\in [0,T]~|~\forall u \in [0,s): \tilde Y_u > \varphi(s)\},
    \]
    then $\tau_0 = \tilde\tau_0$ and $\tilde Y_t = Y_t$ for all $t\in[0,\tau_0]$.
\end{theorem}

\begin{proof}
    The proof is based on careful approximation of the non-Lipschitz drift by some Lipschitz functions. The approximants are explicit and can be used for numerical purposes. Nevertheless, the proof is quite technical and we have set it in the Appendix \ref{sec: existence of local solution}.
\end{proof}

Theorem \ref{thm: existence of the local solution} shows that equation \eqref{volatility1} has a unique solution until the latter stays above $\{\varphi(t), t\in [0,T]\}$. However, an additional condition \textbf{(A4)} on the constant $\gamma$ from assumption \textbf{(A3)} allows to ensure that the corresponding process $Y$ always stays above $\varphi$. More precisely, we have the following result.

\begin{theorem}\label{th: below-bounded solution exists globally}
    Let assumptions \textbf{(A1)}--\textbf{(A4)} hold. Then the process $Y$ introduced in \eqref{volatility1} satisfies
    \[
        Y_t > \varphi(t), \quad t\in [0,T],
    \]
    and therefore the equation \eqref{volatility1} has a unique solution on the entire $[0,T]$. 
\end{theorem}

\begin{proof}
    Assume that $\tau := \inf\{t \in[0,T]~|~ Y_t = \varphi(t)\} \in [0,T]$ (here we assume that $\inf\emptyset = +\infty$). For any $\varepsilon < \min\left\{y_*, Y_0 - \varphi(0)\right\}$, where $y_*$ is from assumption \textbf{(A3)}, consider
    \[
        \tau_\varepsilon := \sup\{t\in[0,\tau]~|~Y_t = \varphi(t) + \varepsilon\}.
    \]
    Due to the definitions of $\tau$ and $\tau_\varepsilon$, 
    \[
        \varphi(\tau) - \varphi(\tau_\varepsilon) -\varepsilon = Y_{\tau} - Y_{\tau_\varepsilon} = \int^{\tau}_{\tau_\varepsilon} b(s, Y_s) ds + Z_{\tau} - Z_{\tau_\varepsilon}.
    \]
    Moreover, for all $t\in[\tau_\varepsilon, \tau)$: $(t,Y_t) \in \mathcal D_{0} \setminus \mathcal D_{\varepsilon}$, so, using the fact that $\varepsilon < y_*$ and assumption \textbf{(A3)}, we obtain that for $t\in[\tau_\varepsilon, \tau)$:
    \begin{equation}\label{PositiveThetaPlaysARoleHere}
        b(t, Y_t) \ge \frac{c}{(Y_t - \varphi(t))^\gamma} \ge  \frac{c}{\varepsilon^\gamma}.
    \end{equation}
    Finally, due to the H\"older continuity of $\varphi$ and $Z$,
    \[
        -(Z_{\tau} - Z_{\tau_\varepsilon}) + (\varphi(\tau) - \varphi(\tau_\varepsilon)) \le (\Lambda + K) (\tau - \tau_\varepsilon)^{\lambda} =: \bar\Lambda (\tau - \tau_\varepsilon)^{\lambda}.
    \]
    Therefore, taking into account all of the above, we get:
    \[
        \bar\Lambda (\tau - \tau_\varepsilon)^{\lambda } \ge \int^{\tau}_{\tau_\varepsilon} \frac{c}{\varepsilon^\gamma} ds + \varepsilon = \frac{c(\tau - \tau_\varepsilon)}{\varepsilon^\gamma} + \varepsilon,
    \]
    i.e.
    \begin{equation}\label{ZeoHittingImpossible}
        \frac{c(\tau - \tau_\varepsilon)}{\varepsilon^\gamma} - \bar\Lambda (\tau - \tau_\varepsilon)^{\lambda}  + \varepsilon \le 0.
    \end{equation}
    Now consider the function $F_\varepsilon$: $\mathbb R^+ \to \mathbb R$ such that
    \[
        F_\varepsilon (t) = \frac{c}{\varepsilon^\gamma} t - \bar\Lambda t^{\lambda}  + \varepsilon.
    \]
    According to \eqref{ZeoHittingImpossible}, $F_\varepsilon(\tau - \tau_\varepsilon) \le 0$ for any $0 < \varepsilon < \min\left\{y_*, Y_0 - \varphi(0)\right\}$. It is easy to verify that $F_\varepsilon$ attains its minimum at the point 
    \[
        t^* = \left(\frac{ \lambda \bar\Lambda}{c}\right)^{\frac{1}{1-\lambda}}\varepsilon^{\frac{\gamma}{1-\lambda}}
    \]
    and
    \[
        F_\varepsilon (t^*) =\varepsilon - D \bar \Lambda^{\frac{1}{1 - \lambda}} \varepsilon^{\frac{\gamma\lambda}{1-\lambda}},
    \]
    where $D :=\left(\frac{1}{c}\right)^{\frac{\lambda}{1 - \lambda}} \left(  \lambda^{\frac{\lambda}{1 - \alpha}} - \lambda^{\frac{1}{1 - \lambda}} \right) > 0$.
    Note that, by \textbf{(A4)}, we have $\frac{\gamma\lambda}{1-\lambda}>1$. Hence it is easy to verify that there exists $\varepsilon^*$ such that for all $\varepsilon < \varepsilon^*$ $F_\varepsilon (t^*)>0$, which contradicts \eqref{ZeoHittingImpossible}. Therefore, $\tau$ cannot belong to $[0,T]$ and $Y$ exceeds $\varphi$.
\end{proof}


\begin{remark}\hspace{10cm}
\begin{enumerate}
\item The result above can be generalized to the case of infinite time horizon in a straightforward manner. For this, it is sufficient to assume that $\varphi$ is {locally} $\lambda$-H\"older continuous, $Z$ has {locally} H\"older continuous paths, i.e. for each $T>0$ there exist constant $K_T >0$ and random variable $\Lambda = \Lambda_T(\omega) > 0$ such that 
\[
|\varphi(t) - \varphi(s)| \le K_T|t-s|^\lambda, \quad |Z_t - Z_s| \le \Lambda_T |t-s|^\lambda, \quad t,s\in[0,T],
\]
and assumptions \textbf{(A1)}--\textbf{(A4)} hold on $[0,T]$ for any $T>0$ (in such case, constants $c_\varepsilon$, $y_*$ and $c$ from the corresponding assumptions are allowed to depend on $T$).
\item Since all the proofs above are based on pathwise calculus, it is possible to extend the results to stochastic $\varphi$ and $Y_0$ (provided that $Y_0 > \varphi(0)$). 
\end{enumerate}
\end{remark}

\subsection{Upper and lower bounds for the solution}

As we have seen in the previous subsection, each random variable $Y_t$, $t\in[0,T]$, is \emph{a priori} lower sandwiched by the deterministic value $\varphi(t)$ (under assumptions \textbf{(A1)}--\textbf{(A4)}). In this subsection, we derive additional bounds from above and below for $Y_t$ in terms of the random variable $\Lambda$ characterizing the noise from \eqref{eq: Holder continuity of Z}. Furthermore, such bounds allow us to establish existence of moments of $Y$ of all orders including the negative ones.

\begin{theorem}\label{upper bound for Y}
    Let assumptions \textbf{(A1)}--\textbf{(A4)} hold and $\Lambda$ be the random variable such that
    \[
        |Z_t - Z_s| \le \Lambda |t-s|^\lambda, \quad t,s\in[0,T].
    \] 
    Then, for any $r>0$,
    \begin{itemize}
        \item[1)] there exist positive deterministic constants $M_1(r, T)$ and $M_2(r, T)$ such that
        \[
            |Y_t|^r \le M_1(r, T) + M_2(r, T) \Lambda^r, \quad t\in [0,T];
        \]
        \item[2)] additionally, if $\Lambda$ can be chosen in such a way that $\mathbb E \Lambda^r < \infty$, then
        \[
            \mathbb E \left[ \sup_{t\in [0,T]} |Y_t|^r \right] < \infty.
        \]
    \end{itemize}
\end{theorem}

\begin{proof}
    It is enough to prove item 1) for $r=1$ as the rest of the Theorem will then become clear. Denote $\eta := \frac{Y(0) - \varphi(0)}{2}$ and let
    \[
        \tau_1 := \sup\left\{ s\in [0,T]~|~ \forall u \in [0,s]: Y_u \ge \varphi(u) + \eta\right\}.
    \]
    Our initial goal is to prove the inequality of the form
    \begin{equation}\label{eq: pre-Gronwall for Y}
    \begin{aligned}
        \left|Y_t\right| \le |Y_0| + T A_T + A_T \int_0^t |Y_s| ds + \Lambda T^\lambda + \max_{u\in[0,T]}|\varphi(u)| + \eta,
    \end{aligned}
    \end{equation}
    where 
    \[
        A_T := c_\eta\left(1+ \max_{u\in[0,T]} |\varphi(u)| + \eta \right) + \max_{u \in[0,T]}\left|b\left(u, \varphi(u) + \eta\right)\right|
    \]
    and $c_\eta$ is from assumption \textbf{(A2)}. 

    Similarly to Proposition \ref{finiteness of Y infty} in Appendix \ref{sec: existence of local solution}, we get \eqref{eq: pre-Gronwall for Y} considering the cases $t\le \tau_1$ and $t > \tau_1$ separately.
    \\
    \textbf{Case } $\boldsymbol{t\le \tau_1}$. For any $s\in[0,t]$: $(s, Y_s) \in \mathcal D_{\eta}$ and, therefore, by assumption \textbf{(A2)}, for all $s\in [0,t]$: 
    \[
        \left|b(s, Y_s) - b\left(s, \varphi(s) + \eta\right)\right| \le c_\eta \left|Y_s - \varphi(s) - \eta\right|,
    \]
    hence
    \begin{align*}
        |b(s,Y_s)| &\le c_\eta |Y_s| + c_\eta\left(\max_{u\in[0,T]} |\varphi(u)| + \eta \right) + \max_{u \in[0,T]}\left|b\left(u, \varphi(u) + \eta\right)\right|
        \\
        &\le A_T (1 + |Y_s|).
    \end{align*}
    Therefore, taking into account that $|Z_t| \le \Lambda T^\lambda$, we have:
    \begin{align*}
        \left|Y_t\right| & = \left|Y_0 + \int_0^t b(s, Y_s) ds + Z_t\right|
        \\
        &\le |Y_0| + \int_0^t |b(s, Y_s)|ds + |Z_t|
        \\
        &\le |Y_0| + TA_T + A_T \int_0^t |Y_s| ds + \Lambda T^\lambda
        \\
        &\le |Y_0| + T A_T + A_T \int_0^t |Y_s| ds + \Lambda T^\lambda + \max_{u\in[0,T]}|\varphi(u)| + \eta.
    \end{align*}
    \textbf{Case} $\boldsymbol{t > \tau_1}$. From the definition of $\tau_1$ and continuity of $Y$, $Y_{\tau_1} = \eta$. Furthermore, since $Y_s > \varphi(s)$ for all $s \ge 0$, we can consider
    \[
        \tau_2(t) := \sup\left\{s \in (\tau_1, t]~|~Y_s < \varphi(s) + \eta\right\}.
    \]
    Note that $\left|Y_{\tau_2(t)}\right| \le \max_{u\in[0,T]}|\varphi(u)| + \eta$, so
    \begin{equation}\label{eq: splitting of Y}
    \begin{aligned}
        \left| Y_t \right| &\le \left| Y_t - Y_{\tau_2(t)} \right| + \left|Y_{\tau_2(t)}\right|
        \\
        & \le \left| Y_t - Y_{\tau_2(t)} \right| + \max_{u\in[0,T]}|\varphi(u)| + \eta.
    \end{aligned}
    \end{equation}
    If $\tau_2(t) < t$, we have that $(s,Y_s) \in \mathcal D_{\eta}$ for all $s \in [\tau_2(t), t]$, therefore, similarly to Step 1,
    \[
        |b(s, Y_s)| \le A_T(1 + |Y_s|),
    \]
    so
    \begin{equation*}
    \begin{aligned}
        \left| Y_t - Y_{\tau_2(t)} \right| &= \left|\int_{\tau_2(t)}^t b(s,Y_s)ds + (Z_t - Z_{\tau_2(t)})\right|
        \\
        &\le \int_{\tau_2(t)}^t |b(s,Y_s)|ds + |Z_t - Z_{\tau_2(t)}|
        \\
        &\le T A_T + A_T \int_0^t |Y_s| ds + \Lambda T^\lambda,
    \end{aligned}
    \end{equation*}
    whence, taking into account \eqref{eq: splitting of Y}, we have:
    \begin{equation}
    \begin{aligned}
        \left| Y_t \right| &\le T A_T + A_T \int_0^t |Y_s| ds + \Lambda T^\lambda + \max_{u\in[0,T]}|\varphi(u)| + \eta
        \\
        &\le |Y_0| + T A_T + A_T \int_0^t |Y_s| ds + \Lambda T^\lambda + \max_{u\in[0,T]}|\varphi(u)| + \eta.
    \end{aligned}
    \end{equation}
    Now, when we have seen that \eqref{eq: pre-Gronwall for Y} holds for any $t\in [0,T]$, we apply the Gronwall's inequality to get
    \begin{equation*}
    \begin{aligned}
        |Y_t| &\le \left(|Y_0| + T A_T + \Lambda T^\lambda + \max_{u\in[0,T]}|\varphi(u)| + \eta\right)e^{TA_T}
        \\
        &=: M_1(1, T) + M_2(1, T) \Lambda,
    \end{aligned}
    \end{equation*}
    where
    \begin{align*}
        M_1(1, T) &:= \left(|Y_0| + T A_T + \max_{u\in[0,T]}|\varphi(u)| + \frac{Y_0 - \varphi(0)}{2}\right)e^{TA_T},
        \\
        M_2(1, T) &:= T^\lambda e^{TA_T}.
    \end{align*}
\end{proof}

Hereafter we provide further specifications on the solution bounds. For this, recall that, for all $t,s\in[0,T]$, the lower sandwich function $\varphi$ satisfies
\[
    |\varphi(t) - \varphi(s)| \le K|t-s|^\lambda
\]
and the noise satisfies \textbf{(Z2)}:
\[
    |Z_t - Z_s| \le \Lambda|t-s|^\lambda,
\]
where $K>0$ is a deterministic constant and $\Lambda$ is a positive random variable.

\begin{theorem}\label{lower bound for volatility}
    Let assumptions \textbf{(A1)}--\textbf{(A4)} hold and $\Lambda$ be the random variable such that
    \[
        |Z_t - Z_s| \le \Lambda |t-s|^\lambda, \quad t,s\in[0,T].
    \] 
    Then, for any $r>0$, 
    \begin{itemize}
        \item[1)] there exists a constant $M_3(r,T) >0$ depending only on $r$, $T$, $\lambda$, $\gamma$ and the constant $c$ from assumption \textbf{(A3)} such that for all $t\in[0,T]$:
    \begin{equation}\label{eq: lower bound for Y, in formulation of the theorem}
        (Y_t - \varphi(t))^{-r} \le  M_3(r,T) { \tilde \Lambda ^{\frac{r }{\gamma \lambda + \lambda -1}} }, 
    \end{equation}
    where
    \[
        \tilde\Lambda := \max\left\{ \Lambda, K, \left(2 \beta\right)^{\lambda - 1} \left(\frac{(Y_0 - \varphi(0)) \wedge y_*}{2}\right)^{1 - \lambda - \gamma\lambda}  \right\}
    \]
    with 
    \[
        \beta := \frac{    \lambda^{\frac{\lambda}{1 - \lambda}} -  \lambda^{\frac{1}{1 - \lambda}} }{ c ^{\frac{\lambda}{1 - \lambda}}} > 0;
    \]
    \item[2)] additionally, if $\Lambda$ can be chosen in such a way that $\mathbb E \Lambda^{\frac{r}{\gamma\lambda + \lambda - 1}} < \infty$, then
        \[
            \mathbb E \left[ \sup_{t\in [0,T]} (Y_t - \varphi(t))^{-r} \right] < \infty.
        \]
    \end{itemize}
\end{theorem}

\begin{proof}
    Just as in Theorem \ref{upper bound for Y}, it is enough to prove that there exists a constant $L>0$ that depends only on $T$, $\lambda$, $\gamma$ and the constant $c$ from assumption \textbf{(A3)} such that for all $t\in[0,T]$
    \[
        Y_t - \varphi(t) \ge \frac{L}{\tilde\Lambda^{\frac{1}{\gamma\lambda + \lambda -1}}}
    \]
    and then the rest of the Theorem will follow.
    
    Put
    \[
        \varepsilon = \varepsilon(\omega) := \frac{1}{(2 \beta )^{\frac{1 - \lambda}{\gamma \lambda + \lambda -1}} \tilde \Lambda ^{\frac{1 }{\gamma \lambda + \lambda -1}}}.
    \]
    Note that $\tilde\Lambda$ is chosen is such a way that
    \[
        |\varphi(t) - \varphi(s)| + |Z_t - Z_s| \le \tilde\Lambda |t-s|^{\lambda}, \quad t, s \in[0,T],
    \]
    and, furthermore, $\varepsilon < Y_0 - \varphi(0)$ and $\varepsilon < y_*$. Fix an arbitrary $t\in[0,T]$. If $Y_t -\varphi(t) > \varepsilon$, then, by definition of $\varepsilon$, estimate of the type \eqref{eq: lower bound for Y, in formulation of the theorem} holds automatically. If $Y_t - \varphi(t) < \varepsilon$, then, since $Y_0 - \varphi(0) > \varepsilon$, one can define
    \[
        \tau(t) := \sup\{s\in[0,t]~|~Y_s - \varphi(s) = \varepsilon\}.
    \]
    Since for all $s\in[\tau(t), t]$ $Y_s - \varphi(s) \le \varepsilon < y_*$, one can apply Assumption \textbf{(A3)} and write
    \begin{align*}
        Y_t - \varphi(t) &= Y_{\tau(t)} - \varphi(t) + \int_{\tau(t)}^t b(s, Y_s) ds + Z_t - Z_{\tau(t)}
        \\
        &= \varepsilon + \varphi(\tau(t)) - \varphi(t) + \int_{\tau(t)}^t b(s, Y_s) ds + Z_t - Z_{\tau(t)}
        \\
        &\ge \varepsilon + \frac{ c}{\varepsilon^{\gamma}}(t - \tau(t)) - \tilde\Lambda (t - \tau(t))^{\lambda}.
    \end{align*}
    Consider the function $F_{\varepsilon} : \mathbb R_+ \to \mathbb R$ such that
    \[
        F_{\varepsilon} (x) = \varepsilon + \frac{ c}{\varepsilon^{\gamma}} x - \tilde\Lambda x^{\lambda}.
    \]
    It is straightforward to verify that $F_\varepsilon$ attains its minimum at
    \[
        x_* := \left(\frac{\lambda}{ c}\right)^{\frac{1}{1 - \lambda}} \varepsilon^{\frac{\gamma}{1 - \lambda}} \tilde \Lambda^{\frac{1}{1 - \lambda}}
    \]
    and, taking into account the explicit form of $\varepsilon$,
    \begin{align*}
        F_\varepsilon(x_*) &= \varepsilon+ \frac{\lambda^{\frac{1}{1 - \lambda}}}{ c^{\frac{\lambda}{1 - \lambda}}} \varepsilon^{\frac{\gamma \lambda}{1 - \lambda}} \tilde \Lambda^{\frac{1}{1 - \lambda}} - \frac{\lambda^{\frac{\lambda}{1 - \lambda}}}{c^{\frac{\lambda}{1 - \lambda}}} \varepsilon^{\frac{\gamma \lambda}{1 - \lambda}} \tilde \Lambda^{\frac{1}{1 - \lambda}}
        \\
        &= \varepsilon - \beta \varepsilon^{\frac{\gamma \lambda}{1 - \lambda}} \tilde \Lambda^{\frac{1}{1 - \lambda}}
        \\
        &= \frac{1}{2^{\frac{\gamma\lambda }{\gamma \lambda + \lambda -1}} \beta^{\frac{1 - \lambda}{\gamma \lambda + \lambda -1}} \tilde \Lambda ^{\frac{1 }{\gamma \lambda + \lambda -1}}} 
        \\
        & = \frac{\varepsilon}{2},
    \end{align*}
    i.e., if $Y(t) < \varphi(t) + \varepsilon$, we have that
    \[
        Y(t)  - \varphi(t) \ge F_{\varepsilon} ( t - \tau(t) ) \ge F_\varepsilon(x_*) = \frac{\varepsilon}{2},
    \]
    and thus for any $t\in[0,T]$
    \[
        Y(t)  \ge  \varphi(t) + \frac{\varepsilon}{2} = \varphi(t) + \frac{1}{2^{\frac{\gamma\lambda }{\gamma \lambda + \lambda -1}} \beta^{\frac{1 - \lambda}{\gamma \lambda + \lambda -1}} \tilde\Lambda ^{\frac{1 }{\gamma \lambda + \lambda -1}}} =: \frac{L}{\tilde\Lambda^{\frac{1 }{\gamma \lambda + \lambda -1}}},
    \]
    where $L := \frac{1}{2^{\frac{\gamma\lambda }{\gamma \lambda + \lambda -1}} \beta^{\frac{1 - \lambda}{\gamma \lambda + \lambda -1}} }$. By this, the proof is complete.
\end{proof}

\begin{remark}
    As one can see, the existence of moments for $Y$ comes down to existence of moments for $\Lambda$. Note that the noises given in Examples \ref{ex: Gaussian Holder processes} and \ref{ex: non-Gaussian continuous martingales} fit into this framework.
\end{remark}

\begin{remark}\label{rem: explicit form of constant in lower bound}
    The constant $M_3(r,T)$ from Theorem \ref{lower bound for volatility} can be explicitly written as
    \[
        M_3(r,T) = 2^{\frac{r\gamma\lambda }{\gamma \lambda + \lambda -1}} \beta^{\frac{r(1 - \lambda)}{\gamma \lambda + \lambda -1}}.
    \]
\end{remark}

\section{SDE with sandwiched solution case }\label{sec: two-sided bounds}

The fact that, under assumptions \textbf{(A1)}--\textbf{(A4)}, the solution $Y$ of \eqref{volatility1} stays above the function $\varphi$ is essentially based on the rapid growth to infinity of $b(t, Y_t)$ whenever $Y_t$ approaches $\varphi(t)$, $t\ge 0$. The same effect is exploited in the case of assumptions \textbf{(B1)}--\textbf{(B4)} and the corresponding solution turns out to be both upper and lower bounded, i.e. \emph{sandwiched}.

Recall that $\varphi$, $\psi$: $[0,T] \to \mathbb R$, $\varphi(t) < \psi(t)$, $t \in [0,T]$, are $\lambda$-H\"older continuous functions, $\lambda\in(0,1)$. Consider a stochastic differential equation of the form \eqref{volatility1} with $\varphi(0) < Y_0 < \psi(0)$, $Z$ being, as before, a stochastic process with $\lambda$-H\"older continuous trajectories and the drift $b$ satisfying assumptions \textbf{(B1)}--\textbf{(B4)}.

In line with the previous section, we show that the solution exists and it is sandwiched.
\begin{theorem}\label{th: exisence and uniqueness of sandwiched Y}
    Let assumptions \textbf{(B1)}--\textbf{(B4)} hold. Then the equation \eqref{volatility1} has a unique solution $Y = \{Y_t,~t\in [0,T]\}$ such that 
    \begin{equation}\label{eq: sandwichness of Y}
        \varphi(t) < Y_t <\psi(t), \quad t\in[0,T].
    \end{equation}
\end{theorem}

\begin{proof}
    The proof uses the techniques presented in Appendix \ref{sec: existence of local solution} and  section \ref{sec: one-sided bound} in a straightforward manner, so full details will be omitted. Here we present only the kernel points.

    First, let $n_0 > - \min_{t\in[0,T]} b(t, \psi(t) - y_*)$, with $y_*$ being from assumption \textbf{(B3)}. For an arbitrary $n \ge n_0$ define the set
    \[
        \widehat{\mathcal G}_n := \left\{(t,y) \in {\mathcal D}_{0,0} \setminus {\mathcal D}_{0,y_*}  ~|~b(t,y) > - n  \right\},
    \]
    and consider the stochastic process $Y^{(n)}_t$ that is the solution to the stochastic differential equation of the form
    \[
        dY^{(n)}_t = \hat{f}_n(t, Y^{(n)}_t)dt + dZ_t, \quad Y^{(n)}_0 = Y_0>0,
    \]
    where 
    \[
        \hat{f}_n(t,y) = \begin{cases}
            b(t,y) + \frac{1}{n}, &\quad (t,y)\in \widehat{\mathcal G}_n \cup {\mathcal D}_{0,y_*},
            \\
            -n + \frac{1}{n}, &\quad (t,y)\in [0,T]\times \mathbb R \setminus \left(\widehat{\mathcal G}_n \cup {\mathcal D}_{0,y_*}\right).
        \end{cases}
    \]
    Observe that each $\hat f_n$ satisfies assumptions \textbf{(A1)}--\textbf{(A4)}. Therefore, by Theorem \ref{th: below-bounded solution exists globally}, each $\{Y^{(n)}_t,~t\in[0,T]\}$, $n\ge n_0$, exists, is unique and exceeds $\{\varphi(t),~t\in[0,T]\}$. Furthermore, by the virtue of Theorem \ref{lower bound for volatility}, 
    \[
        Y^{(n)}_t \ge \varphi(t) + \xi, \quad t\in[0,T],
    \]
    where $\xi < y_*$ is a positive random variable that does not depend on $n$. In other words, each $\{Y^{(n)}_t,~t\in[0,T]\}$ is, in fact, a unique solution to the equation
    \[
        dY^{(n)}_t = \hat{b}_n(t, Y^{(n)}_t)dt + dZ_t, \quad Y^{(n)}_0 = Y_0,
    \]
    with
    \[
        \hat{b}_n (t,y) = \begin{cases}
            \hat{f}_n(t,y),&\quad (t,y) \in {\mathcal D}_{\xi},
            \\
            b(t, \varphi(t) + \xi) + \frac{1}{n}, &\quad (t,y) \in [0,T]\times \mathbb R \setminus {\mathcal D}_{\xi}.
        \end{cases} 
    \]
    Now, following Appendix \ref{sec: existence of local solution}, it is easy to verify that $Y_t = Y^{(\infty)}_t = \lim_{n\to\infty} Y^{(n)}_t$, $t\in[0,T]$, is correctly defined and is a unique stochastic process that satisfies \eqref{volatility1} until the first moment of crossing $\psi(t)$, $t\in[0,T]$. The given claim follows by the argument similar to the one in Theorem \ref{th: below-bounded solution exists globally}.
\end{proof}

Similarly to Theorem \ref{lower bound for volatility}, the bounds \eqref{eq: sandwichness of Y} obtained in the existence--uniqueness Theorem \ref{th: exisence and uniqueness of sandwiched Y} can be refined.

\begin{theorem}\label{th: moments of sandwiched Y}
    Let $r>0$ be fixed.
    \begin{enumerate}
        \item Under conditions \textbf{(B1)}--\textbf{(B4)}, there exists a constant $L >0$ depending only on $\lambda$, $\gamma$ and the constant $c$ from assumption \textbf{(B3)} such that the solution $Y$ to the equation \eqref{volatility1} has the property
        \[
            \varphi(t) + L \tilde\Lambda^{- \frac{1}{\gamma\lambda + \lambda - 1}} ~\le~ Y_t ~\le~ \psi(t) -  L \tilde\Lambda^{- \frac{1}{\gamma\lambda + \lambda - 1}}, \quad t\in[0,T], 
        \]
        where
        \[
            \tilde\Lambda := \max\left\{ \Lambda, K, \left(4 \beta\right)^{\lambda - 1} \left(\frac{(Y_0 - \varphi(0)) \wedge y_* \wedge (\psi(0) - Y_0)}{2}\right)^{1 - \lambda - \gamma\lambda}  \right\}
        \]
        with 
        \[
            \beta := \frac{    \lambda^{\frac{\lambda}{1 - \lambda}} -  \lambda^{\frac{1}{1 - \lambda}} }{ \left(2^{\gamma} c \right)^{\frac{\lambda}{1 - \lambda}}} > 0
        \]
        and $K$ being such that
        \[
            |\varphi(t) - \varphi(s)| + |\psi(t) - \psi(s)| \le K |t-s|^{\lambda}, \quad t,s\in[0,T].
        \]
        \item If $\Lambda$ can be chosen in such a way that $\mathbb E \Lambda^{\frac{r}{\gamma\lambda + \lambda - 1}} < \infty$, then
        \[
            \mathbb E \left[ \sup_{t\in [0,T]} (Y_t - \varphi(t))^{-r} \right] < \infty \quad\text{and}\quad \mathbb E \left[ \sup_{t\in [0,T]} (\psi(t) - Y_t)^{-r} \right] < \infty.
        \]
    \end{enumerate}
\end{theorem}

\begin{proof}
    The proof is similar to the one of Theorem \ref{lower bound for volatility}.
\end{proof}


\section{Applications: generalized CIR and CEV processes}\label{sec: examples}

In this section, we show how two classical processes used in stochastic volatility modeling can be generalized under our framework.

\subsection{CIR and CEV processes driven by a H\"older continuous noise}\label{Ex: CIR-CEV}

Let $\varphi \equiv 0$. Consider
\[
    b(y) = \frac{\kappa}{y^{\frac{\alpha}{1-\alpha}}} - \theta y,
\]
where $\kappa$, $\theta >0$ are positive constants, $\alpha \in\left[\frac{1}{2}, 1\right)$, and the process $Z$ is a process with $\lambda$-H\"older continuous paths with $\alpha + \lambda >1$. It is easy to verify that for $\gamma = \frac{\alpha}{1-\alpha}$ assumptions \textbf{(A1)}--\textbf{(A4)} hold and the prosess $Y$ satisfying the stochastic differential equation
\begin{equation}\label{CIR-CEV}
    dY_t = \left(\frac{\kappa}{Y_t^{\frac{\alpha}{1-\alpha}}} - \theta Y_t\right)dt + dZ_t
\end{equation}
exists, is unique and positive. Furthermore, as it is noted in Theorems \ref{upper bound for Y} and \ref{lower bound for volatility}, if the corresponding H\"older continuity constant $\Lambda$ can be chosen to have all positive moments, $Y$ will have moments of all real orders, including the negative ones.

The process $X = \{X_t,~t\in[0,T]\}$  such that
\[
    X_t = Y_t^{\frac{1}{1 - \alpha}}, \quad t\in [0,T],
\]
can be interpreted as a generalization of CIR (if $\alpha = \frac{1}{2}$) or CEV (if $\alpha \in \left(\frac{1}{2}, 1\right)$) process in the following sense. Assume that $\lambda > \frac{1}{2}$. Fix the partition $0 = t_0 < t_1 < t_2 < ... < t_n =  t$ where $t\in[0,T]$, $\lvert \Delta t \rvert := \max_{k=1,...,n}(t_{k}-t_{k-1})$. It is clear that
\[
    X_t = X_0 + \sum_{k=1}^n (X_{t_k} - X_{t_{k-1}}) = X_0 + \sum_{k=1}^n (Y^\frac{1}{1-\alpha}_{t_k} - Y^\frac{1}{1-\alpha}_{t_{k-1}}),
\]
so, using the Taylor's expansion, we obtain that
\[
    X_t = X_0 + \sum_{k=1}^n \left( \frac{1}{1-\alpha} Y^\frac{\alpha}{1-\alpha}_{t_{k-1}} (Y_{t_k} - Y_{t_{k-1}}) + \frac{\alpha \Theta_k^{\frac{2\alpha-1}{1-\alpha}}}{2(1-\alpha)^2}(Y_{t_k} - Y_{t_{k-1}})^2\right)
\]
with $\Theta_k$ being a real value between $Y_{t_k}$ and $Y_{t_{k-1}}$. 

Using equation \eqref{CIR-CEV} and Theorem \ref{lower bound for volatility}, it is easy to prove that $Y$ has trajectories which are $\lambda$-H\"older continuous, therefore, since $\lambda>\frac{1}{2}$,
\begin{equation}
    \sum_{k=1}^n \frac{\lambda \Theta_k^{\frac{2\alpha-1}{1-\alpha}}}{2(1-\alpha)^2}(Y_{t_k} - Y_{t_{k-1}})^2 \to 0, \quad \lvert \Delta t\rvert \to 0,
\end{equation}
and
\begin{equation}\label{ConvToEquation}
\begin{gathered}
    \sum_{k=1}^n \frac{1}{1-\alpha} Y^\frac{\alpha}{1-\alpha}_{t_{k-1}} (Y_{t_k} - Y_{t_{k-1}}) = \frac{1}{1-\alpha} \sum_{k=1}^n  X^\alpha_{t_{k-1}} (Y_{t_k} - Y_{t_{k-1}})
    \\
    = \frac{1}{1-\alpha} \sum_{k=1}^n  X^\alpha_{t_{k-1}}\left(\int_{t_{k-1}}^{t_k} \left(\frac{\kappa}{Y_s^{\frac{\alpha}{1-\alpha}}} - \theta Y_s\right)ds + (Z_{t_{k}} - Z_{t_{k-1}})\right)
    \\
    = \frac{1}{1-\alpha} \sum_{k=1}^n  X^\alpha_{t_{k-1}}\int_{t_{k-1}}^{t_k} \left(\frac{\kappa}{X_s^\alpha} - \theta X_s^{1-\alpha}\right)ds + \frac{1 }{1-\alpha} \sum_{k=1}^n  X^\alpha_{t_{k-1}}(Z_{t_{k}} - Z_{t_{k-1}})
    \\
    \to \frac{1}{1-\alpha} \int_0^t (\kappa - \theta X_s)ds + \frac{1 }{1-\alpha} \int_0^t X_s^\alpha dZ_s, \quad \lvert \Delta t \rvert \to 0.
\end{gathered}
\end{equation}
Note that the integral with respect to $Z$ in \eqref{ConvToEquation} exists as a pathwise limit of Riemann-Stieltjes integral sums due to sufficient H\"older continuity of both the integrator and integrand, see e.g. \cite{Zahle1998}. 

Taking into account all of the above, the $X$ satisfies (pathwisely) the stochastic differential equation of the CIR (or CEV) type, namely
\begin{equation}\label{eq: generalized CIR CEV equation}
    dX_t = \left(\frac{\kappa}{1-\alpha} - \frac{\theta}{1-\alpha} X_t\right)dt + \frac{1}{1-\alpha} X_t^\alpha dZ_t = (\tilde\kappa - \tilde\theta X_t)dt + \tilde\nu X^\alpha_t dZ_t,
\end{equation}
where the integral with respect to $Z$ is the pathwise Riemann-Stieltjes integral.

\begin{remark}
    The integral $\int_0^t X_s^\alpha dZ_s$ arising above is a pathwise Young integral, see e.g. \cite[Section 4.1]{FH2014} and references therein. 
\end{remark}

\begin{remark}
    Note that the reasoning described above also implies that, for $\alpha\in\left[\frac{1}{2}, 1\right)$ and $\lambda>\frac{1}{2}$, the SDE \eqref{eq: generalized CIR CEV equation}, where the integral w.r.t. $Z$ is understood pathwisely, has a unique strong solution in the class of non-negative stochastic processes with $\lambda$-H\"older continuous trajectories. Indeed, $\{Y^{\frac{1}{1-\alpha}}_t,~t\in[0,T]\}$ with $Y$ defined by \eqref{CIR-CEV} is a solution to \eqref{eq: generalized CIR CEV equation}. Moreover, if $X$ is another solution to \eqref{eq: generalized CIR CEV equation}, then by the chain rule \cite[Theorem 4.3.1]{Zahle1998}, the process $\{X^{1-\alpha}_t,~t\in[0,T]\}$ must satisfy the SDE \eqref{CIR-CEV} until the first moment of zero hitting. However, the SDE \eqref{CIR-CEV} has a unique solution that never hits zero and thus $X^{1-\alpha}$ coincides with $Y$.   
\end{remark}

\begin{remark}
Some of the properties of the process $Y$ given by \eqref{CIR-CEV} in the case of $\lambda = \frac{1}{2}$ and $Z$ being a fractional Brownian motion with $H>\frac{1}{2}$ were discussed in \cite{MYuT2018}.
\end{remark}

\subsection{Mixed-fractional CEV-process}

Assume that $\kappa$, $\theta$, $\nu_1$, $\nu_2$ are positive constants,  $B = \{B_t,~t\in[0,T]\}$ is a standard Wiener process, $B^H = \{B^H_t,~t\in[0,T]\}$ is a fractional Brownian motion independent of $B$ with $H\in\left(0,1\right)$, $Z = \nu_1 B + \nu_2 B^H$, $\alpha \in\left(\frac{1}{2}, 1\right)$ is such that $H \wedge \frac{1}{2} + \alpha > 1$ and the function $b$ has the form
\[
    b(y) = \frac{\kappa}{y^{\frac{\alpha}{1-\alpha}}} - \frac{\alpha \nu_1^2}{2y} - \theta y.
\]
Then the process $Y$ defined by the equation
\begin{equation}\label{mixed CEV}
    dY_t = \left(\frac{\kappa}{Y_t^{\frac{\alpha}{1-\alpha}}} - \frac{\alpha \nu_1^2}{2(1-\alpha)Y_t} - \theta Y_t\right)dt + \nu_1 dB_t + \nu_2 dB^H_t
\end{equation}
exists, is unique, positive and has all the moments of real orders.

If $H>\frac{1}{2}$, just as in subsection \ref{Ex: CIR-CEV}, the process $X_t := Y_t^{\frac{1}{1 - \alpha}}$, $t\in [0,T]$, can be interpreted as a generalization of the CEV-process.

\begin{proposition}
    Let $H>\frac{1}{2}$. Then the process $X_t := Y_t^{\frac{1}{1 - \alpha}}$, $t\in [0,T]$, a.s. satisfies the SDE of the form
    \begin{equation}\label{eq: mixed-fractional CEV}
        dX_t = \left(\frac{\kappa}{1-\alpha} - \frac{\theta}{1-\alpha} X_t\right)dt + \frac{\nu_1}{1-\alpha} X_t^\alpha dB_t + \frac{\nu_2}{1-\lambda} X_t^\alpha dB^H_t,
    \end{equation}
    where the integral with respect to $B$ is the regular It\^o integral (w.r.t. filtration generated jointly by $(B, B^H)$) and the integral with respect to $B^H$ is understood as the $L^2$-limit of Riemann-Stieltjes integral sums.
\end{proposition}
   
\begin{proof}
    We will use the argument that is similar to the one presented in subsection \ref{Ex: CIR-CEV} with one main difference: since we are going to treat the integral w.r.t. the Brownian motion $B$ as a regular It\^o integral, all the convergences (including convergence of integral sums w.r.t. $B^H$) must be considered in $L^2$-sense. For reader's convenience, we split the proof into several steps.
    
    \noindent \textbf{Step 1.} First, we will prove that the integral $\int_0^t X_s^\alpha dB^H_s$ is well defined as the $L^2$-limit of Riemann-Stieltjes integral sums. Let $0 = t_0 < t_1 < t_2 < ... < t_n =  t$ be a partition of $[0,t]$ with the mesh $|\Delta t| := \max_{k=0,...,n-1}(t_{k+1} - t_k)$. 
    
    Choose $\lambda \in \left(\frac{1}{2}, H\right)$, $\lambda' \in \left(0,\frac{1}{2}\right)$ and $\varepsilon >0$ such that $\lambda + \lambda' > 1$ and $\lambda + \varepsilon < H$, $\lambda' + \varepsilon < \frac{1}{2}$. Using Theorem \ref{lower bound for volatility} and the fact that for any $\lambda' \in \left(0,\frac{1}{2}\right)$ the random variable $\Lambda_{Z,\lambda'+\varepsilon}$ which corresponds to the noise $Z$ and H\"older order $\lambda'+\varepsilon$ can be chosen to have moments of all orders, it is easy to prove that there exists a random variable $\Upsilon_X$ having moments of all orders such that
    \[
        |X^\alpha_t - X^\alpha_s| \le \Upsilon_X |t-s|^{\lambda' + \varepsilon}, \quad s,t\in[0,T], \quad a.s.
    \]
    
    By the Young-L\'oeve inequality (see e.g. \cite[Theorem 6.8]{FV2010}), it holds a.s. that
    \begin{align*}
        \left|\int_0^t X^\alpha_s dB^H_s - \sum_{k=0}^{n-1} X^\alpha_{t_k}(B^H_{t_{k+1}} - B^H_{t_k})\right| &\le \sum_{k=0}^{n-1} \left|\int_{t_k}^{t_{k+1}} X^\alpha_s dB^H_s - X^\alpha_{t_k}(B^H_{t_{k+1}} - B^H_{t_k})\right| \\
        &\le \frac{1}{2^{1-(\lambda + \lambda')}}\sum_{k=0}^{n-1} [X^\alpha]_{\lambda'; [t_k, t_{k+1}]} [B^H]_{\lambda; [t_k, t_{k+1}]},
    \end{align*}
    where
    \[
        [f]_{\lambda; [t,t']} := \left(\sup_{\Pi[t, t']} \sum_{l=0}^{m-1} |f(s_{l+1}) - f(s_l))|^{\frac{1}{\lambda}}\right)^\lambda,
    \]
    with supremum taken over all partitions $\Pi[t,t'] = \{t=s_0 < ... < s_m = t'\}$ of $[t,t']$.

    It is clear that, a.s.,
    \begin{align*}
        [X^\alpha]_{\lambda'; [t_k, t_{k+1}]} &= \left(\sup_{\Pi[t_k, t_{k+1}]} \sum_{l=0}^{m-1} |X^\alpha(s_{l+1}) - X^\alpha(s_l))|^{\frac{1}{\lambda'}}\right)^{\lambda'}  
        \\
        &\le \Upsilon_X \left(\sup_{\Pi[t_k, t_{k+1}]} \sum_{k=0}^{m-1} (s_{l+1} - s_l)^{1+ \frac{\varepsilon}{\lambda'}}\right)^{\lambda'}
        \\
        &\le \Upsilon_X |\Delta t|^{\lambda'+\varepsilon}
    \end{align*}
    and, similarly,
    \begin{align*}
        [B^H]_{\lambda; [t_k, t_{k+1}]} \le \Lambda_{B^H} |\Delta t|^{\lambda+\varepsilon},
    \end{align*}
    where $\Lambda_{B^H}$ has moments of all orders and
    \[
        |B^H_t - B^H_s| \le \Lambda_{B^H} |t-s|^{\lambda +\varepsilon},
    \]
    whence
    \begin{align*}
        \mathbb E\left|\int_0^t X^\alpha_s dB^H_s - \sum_{k=0}^{n-1} X^\alpha_{t_k}(B^H_{t_{k+1}} - B^H_{t_k})\right|^2 &\le \mathbb E\left[\left(\frac{1}{2^{1-(\lambda + \lambda')}}\sum_{k=0}^{n-1} [X^\alpha]_{\lambda'; [t_k, t_{k+1}]} [B^H]_{\lambda; [t_k, t_{k+1}]}\right)^2\right]
        \\
        &\le \mathbb E\left[\Lambda^2_{B^H} \Upsilon_X^2 \frac{1}{2^{2-2(\lambda + \lambda')}}\left(\sum_{k=0}^{n-1} |\Delta t|^{\lambda + \lambda' + 2\varepsilon}\right)^2\right] \to 0,
    \end{align*}
    as $ |\Delta t|\to 0$. It is now enough to note that each Riemann-Stieltjes sum is in $L^2$ (thanks to the fact that $\mathbb E[\sup_{t\in[0,T]} X^r_t] < \infty$ for all $r>0$), so the integral $\int_0^t X^\alpha_s dB^H_s$ is indeed well-defined as the $L^2$-limit of Riemann-Stieltjes integral sums.

    \noindent\textbf{Step 2.} Now, we would like to get representation \eqref{eq: mixed-fractional CEV}. In order to do that, one should follow the proof of the It\^o formula in a similar manner to subsection \ref{Ex: CIR-CEV}. Namely, for a partition $0 = t_0 < t_1 < t_2 < ... < t_n =  t$ one can write
    \begin{align*}
        X_t &= X_0 + \sum_{k=1}^{n} \left( Y^{\frac{1}{1-\alpha}}_{t_k} - Y^{\frac{1}{1-\alpha}}_{t_{k-1}}\right) 
        \\
        &= X_0 + \frac{1}{1-\alpha} \sum_{k=0}^{n-1} \left(  Y^{\frac{\alpha}{1-\alpha}}_{t_{k-1}} (Y_{t_k} - Y_{t_k-1})\right) + \frac{1}{2} \frac{\alpha}{(1-\alpha)^2} \sum_{k=0}^{n-1} \left(Y^{\frac{2\alpha -1}{1 - \alpha}}_{t_{k-1}} (Y_{t_k} - Y_{t_k-1})^2\right)
        \\
        &\quad+ \frac{1}{6} \frac{\alpha(2\alpha -1)}{(1-\alpha)^3} \sum_{k=1}^n \left( \Theta_k^{\frac{3\alpha - 2}{1-\alpha}}(Y_{t_k} - Y_{t_{k-1}})^3\right),
    \end{align*}
    where $\Theta_k$ is a value between $Y_{t_{k-1}}$ and $Y_{t_{k}}$.
    
    Note that, using Theorems \ref{upper bound for Y} and \ref{lower bound for volatility}, it is easy to check that for any $\lambda'\in\left(\frac{1}{3}, \frac{1}{2}\right)$ there exists a random variable $\Upsilon_Y$ having moments of all orders such that
    \[
        |Y_t - Y_s| \le \Upsilon_Y |t-s|^{\lambda'}.
    \]
    Furthermore, by Theorem \ref{upper bound for Y} (for $\alpha \in \left[\frac{3}{2}, 1\right)$) and Theorem \ref{lower bound for volatility} (for $\alpha \in \left(\frac{1}{2},\frac{3}{2}\right)$), it is clear that there exists a random variable $\Theta >0$ that does not depend on the partition and has moments of all orders such that $\Theta_k < \Theta$, whence
    \[
       \sum_{k=1}^n \left( \Theta_k^{\frac{3\alpha - 2}{1-\alpha}}(Y_{t_k} - Y_{t_{k-1}})^3\right) \le \Theta^{\frac{3\alpha - 2}{1-\alpha}} \Upsilon_Y^3\sum_{k=1}^n (t_k - t_{k-1})^{3\lambda'} \xrightarrow{L^2} 0, \quad |\Delta t| \to 0.
    \]
        Using Step 1, it is also straightforward to verify that
    \begin{align*}
       \frac{1}{1-\alpha} \sum_{k=0}^{n-1} \left(  Y^{\frac{\alpha}{1-\alpha}}_{t_{k-1}} (Y_{t_k} - Y_{t_k-1})\right) \xrightarrow{L^2}& \frac{1}{1-\alpha} \int_0^t \left( \kappa - \theta X_s \right)ds  + \frac{\nu_1}{1-\alpha} \int_0^t X_s^\alpha dB_s 
       \\
       & \quad + \frac{\nu_2}{1-\lambda} \int_0^t X_s^\alpha dB^H_s - \frac{\alpha\nu_1^2}{2(1-\alpha)^2} \int_0^t Y^{\frac{2\alpha -1}{1 - \alpha}}_s ds, \quad |\Delta t| \to 0,
    \end{align*}
    and
    \begin{align*}
       \frac{1}{2} \frac{\alpha}{(1-\alpha)^2} \sum_{k=0}^{n-1} \left(Y^{\frac{2\alpha -1}{1 - \alpha}}_{t_{k-1}} (Y_{t_k} - Y_{t_k-1})^2\right) \xrightarrow{L^2} \frac{\alpha\nu_1^2}{2(1-\alpha)^2} \int_0^t Y^{\frac{2\alpha -1}{1 - \alpha}}_s ds, \quad |\Delta t| \to 0,
    \end{align*}
    which concludes the proof.
\end{proof}

\begin{section}{Semi-heuristic Euler discretization scheme and simulations}\label{sec: approximation schemes}

In this section, we present simulated paths of the sandwiched process based on a semi-heuristic approximation approach. One must note that it does not have the virtue of giving sandwiched discretized process and has worse convergence type in comparison to some alternative schemes (see, for example, \cite{Hong2020, ZhYu2020} for the case of fractional Brownian motion), but, on the other hand, allows much weaker assumptions on both the drift and the noise and is much simpler from the implementation point of view.

\subsection{Numerical scheme and convergence results}

We first consider the setting of assumptions \textbf{(A1)}--\textbf{(A4)}. Additionally, we require local H\"older continuity of the drift $b$ with respect to $t$ in the following sense:
\begin{itemize}
\item[\textbf{(A5)}] for any $\varepsilon > 0$ there is $c_{\varepsilon} > 0$ such that for any $(t,y)$, $(s,y) \in \mathcal D_{\varepsilon}$:
\[
|b(t,y) - b(s, y)| \le c_\varepsilon |t-s|^\lambda.
\]
\end{itemize}
Obviously, without loss of generality one can assume that the constant $c_\varepsilon$ is the same for assumptions \textbf{(A2)} and \textbf{(A5)}.

We stress that the drift $b$ is not globally Lipschitz and, furthermore, for any $t\in[0,T]$, the value $b(t,y)$ is not defined for $y < \varphi(t)$. Hence classical Euler approximations applied directly to the equation \eqref{volatility1} fail since such scheme does not guarantee that the discretized version of the process stays above $\varphi$. 

A straightforward way to overcome this issue is to discretize not the process $Y$ itself, but its approximation $\tilde Y^{(n)}$ obtained by ``leveling'' the singularity in the drift. Namely, let $n_0 > \max_{t\in[0,T]}|b(t, \varphi(t) + y_*)|$, where $y_*$ is from Assumption \textbf{(A3)}. For an arbitrary $n \ge n_0$ define the set
\begin{equation*}
\begin{gathered}
\mathcal G_n := \{(t,y) \in \mathcal D_{0}\setminus\mathcal D_{y_*}~|~b(t,y) < n \}
\end{gathered}
\end{equation*}
and consider the functions $\tilde b_n$: $[0,T]\times\mathbb R \to \mathbb R$ of the form
\begin{equation}\label{eq: bn tilde}
\tilde b_n(t, y) := \begin{cases}
b(t, y), &\quad (t,y) \in \mathcal G_n \cup \mathcal D_{y_*},
\\
n, &\quad (t,y) \in [0,T]\times\mathbb R \setminus \left(\mathcal G_n \cup \mathcal D_{y_*} \right).
\end{cases}
\end{equation}
Denote $\varepsilon_n := \left(\frac{c}{n}\right)^{\frac{1}{\gamma}}$, and observe that 
\begin{equation}\label{eq: Lipschitz condition for approximation}
\begin{gathered}
    |\tilde b_n(t,y_1) - \tilde b_n(t,y_2)| \le c_n|y_1-y_2|, \quad t\in[0,T], \quad y_1,y_2\in \mathbb R,
    \\
    |\tilde b_n(t_1,y) - \tilde b_n(t_2,y)| \le c_n|t_1-t_2|^\lambda, \quad t_1,t_2\in[0,T], \quad y\in \mathbb R,
\end{gathered}
\end{equation}
where $c_n := c_{\varepsilon_n}$ denotes the constant from assumptions \textbf{(A2)} and \textbf{(A5)} which corresponds to $\varepsilon_n = \left(\frac{c}{n}\right)^{\frac{1}{\gamma}}$. In particular, this implies that the SDE
\begin{equation}\label{eq: tilde Yn}
    d\tilde Y^{(n)}_t = \tilde b_n(t,\tilde Y^{(n)}_t)dt + d Z_t, \quad \tilde Y^{(n)}_0 = Y_0 > \varphi(0),
\end{equation}
has a unique pathwise solution which can be approximated by the Euler scheme.

\begin{remark}
    In this section, by $C$ we will denote any positive constant that does not depend on order of approximation $n$ or the the partition and the exact value of which is not important. Note that $C$ may change from line to line (or even within one line). 
\end{remark}

Regarding the process $\tilde Y^{(n)}$, we have the following result.
\begin{proposition}
    Let assumptions \textbf{(A1)}--\textbf{(A4)} hold. Then, for any $r>0$, there exists a constant $C>0$ that does not depend on $n$ such that
    \[
        \max_{t\in[0,T]}|\tilde Y^{(n)}_t|^r \le C\left(1+\left(\max_{t\in[0,T]}|Z_t|\right)^r\right).
    \]
\end{proposition}

\begin{proof}
    First, consider $n=n_0$. It is easy to see that there exists $C>0$ that does not depend $n$ such that
    \[
        |b_{n_0}(t,y)| \le C(1+ |y|),
    \]
    therefore there exists $C>0$ such that
    \begin{align*}
        |\tilde Y^{(n_0)}_t| & \le Y_0 + \int_0^t |b_{n_0}(s, \tilde Y^{(n_0)}_s )|ds + Z_t
        \\
        &\le C+\max_{t\in[0,T]} Z_t + \int_0^t |\tilde Y^{(n_0)}_s|ds,
    \end{align*}
    hence, by Gronwall's inequality,
    \[
        \max_{t\in[0,T]}|\tilde Y^{(n_0)}_t| \le C(1+\max_{t\in[0,T]}|Z_t|)
    \]
    for some constant $C>0$.
    
    Next, the process $\tilde Y^{(n)}$ satisfies the SDE of the form
    \[
        d\tilde Y^{(n)}_t = b_n(t,\tilde Y^{(n)}_t)dt + d\tilde{Z}_t,
    \]
    where $b_n := \tilde{b}_n - \frac{1}{n}$ and $\tilde Z_t := Z_t + \frac{t}{n}$ and thus, similarly to \eqref{eq: finiteness of Y} from Proposition \ref{finiteness of Y infty}, one can verify that
    \begin{align*}
        |\tilde Y^{(n)}_t| &\le C\left(1+ \max_{t\in[0,T]}|\tilde Y^{(n_0)}_t| + \max_{t\in[0,T]}|\tilde Z_t|\right) 
        \\
        &\le C(1+\max_{t\in[0,T]}|Z_t|),
    \end{align*}
    where $C>0$ is some constant that does not depend on $n$. 
\end{proof}

Since 
\[
    |Z_t| = |Z_t-Z_0| \le \Lambda t^\lambda \le T^\lambda \Lambda, \quad t\in [0,T],
\]
we immediately get the following corollary.

\begin{corollary}\label{cor: upper bound for approximations}
    Under assumptions \textbf{(A1)}--\textbf{(A4)}, there exists a constant $C>0$ that does not depend on $n$ such that
    \[
        \max_{t\in[0,T]}\left|\tilde Y^{(n)}_t\right|^r \le C(1+ \Lambda^r).
    \]
\end{corollary}

Before proceeding to the main theorem of the section, let us provide another simple auxiliary proposition.

\begin{proposition}\label{prop: probability of small Y}
    Let assumptions \textbf{(A1)}--\textbf{(A4)} hold. Assume also that the noise $Z$ satisfying Assumptions \textbf{(Z1)}--\textbf{(Z2)} is such that
    \[
        \mathbb E \left[ |Z(t) - Z(s)|^p \right] \le C_{ \lambda, p}|t-s|^{\lambda p}, \quad s,t\in[0,T],
    \]
    for some $p\ge 1$ such that $\lambda_p := \lambda - \frac{2}{p} > \frac{1}{1+\gamma}$ with $\gamma$ from assumption \textbf{(A3)} and a positive constant $C_{\lambda, p} > 0$. Then
    \begin{equation*}
        \mathbb P\left( \min_{t\in[0,T]}(Y(t) - \varphi(t)) \le \varepsilon \right) = O(\varepsilon^{\gamma \lambda_p + \lambda_p - 1}), \quad \varepsilon \to 0.
    \end{equation*}
\end{proposition}

\begin{proof}
    By Lemma \ref{l: GRR},
    \[
        |Z(t) - Z(s)| \le A_{\lambda,p} |t-s|^{\lambda - \frac{2}{p}} \left(\int_0^T \int_0^T \frac{|Z_x - Z_y|^p}{|x-y|^{\lambda p }} dx dy\right)^{\frac{1}{p}},
    \]
    where 
    \begin{equation*}
        A_{\lambda, p} = 2^{3 + \frac{2}{p}}\left( \frac{\lambda p }{\lambda p - 2} \right).
    \end{equation*}
    Note that the random variable
    \[
        \Lambda_p := A_{\lambda,p} \left(\int_0^T \int_0^T \frac{|Z_x - Z_y|^p}{|x-y|^{\lambda p }} dx dy\right)^{\frac{1}{p}}
    \]
    is finite a.s. since
    \begin{align*}
        \mathbb E \Lambda^p_p &= A^p_{\lambda,p} \int_0^T \int_0^T \frac{\mathbb E |Z_x - Z_y|^p}{|x-y|^{\lambda p }} dx dy
        \\
        &\le T^2 A^p_{\lambda,p} C_{\lambda, p}
        \\
        &<\infty.
    \end{align*}
    
    Now, by applying Theorem \ref{lower bound for volatility} and Remark \ref{rem: explicit form of constant in lower bound} with respect to the H\"older order $\lambda_p = \lambda-\frac{2}{p}$, one can deduce that for all $t\in[0,T]$
    \[
        Y(t) - \varphi(t) \ge \frac{1}{M_{3,p}(1,T) \tilde \Lambda_p^{\frac{1}{\gamma \lambda_p + \lambda_p -1}} },
    \]
    where
    \[
        M_{3,p}(1,T) := 2^{\frac{\gamma\lambda_p }{\gamma \lambda_p + \lambda_p -1}} \beta^{\frac{1 - \lambda_p}{\gamma \lambda_p + \lambda_p -1}} >0,
    \]
    \[
        \beta_p := \frac{    \lambda_p^{\frac{\lambda_p}{1 - \lambda_p}} -  \lambda_p^{\frac{1}{1 - \lambda_p}} }{ c ^{\frac{\lambda_p}{1 - \lambda_p}}} > 0
    \]
    and
    \[
        \tilde \Lambda_p := \max\left\{ \Lambda_p, K_p, \left(2 \beta_p\right)^{\lambda_p - 1} \left(\frac{(Y_0 - \varphi(0)) \wedge y_*}{2}\right)^{1 - \lambda_p - \gamma\lambda_p}  \right\}
    \]
    with $y_*$, $c$ and $\gamma$ being from assumption \textbf{(A3)} and $K_p$ being such that
    \[
        |\varphi(t) - \varphi(s)| \le K_p|t-s|^{\lambda_p}, \quad s,t\in[0,T].
    \]
    Therefore
    \begin{align*}
        \mathbb P\left( \min_{t\in[0,T]}(Y(t) - \varphi(t)) \le \varepsilon \right) &\le \mathbb P\left( \frac{1}{M_{3,p}(1,T) \tilde \Lambda_p^{\frac{1}{\gamma \lambda_p + \lambda_p -1}} } \le \varepsilon \right)
        \\
        &= \mathbb P\left( \tilde \Lambda_p \ge \left(\frac{1}{M_{3,p}(1,T)\varepsilon}\right)^{\gamma \lambda_p + \lambda_p -1} \right)
        \\
        &\le (M_{3,p}(1,T))^{\gamma \lambda_p + \lambda_p -1} \mathbb E[\tilde\Lambda_p] \varepsilon^{\gamma \lambda_p + \lambda_p -1}
        \\
        &=O(\varepsilon^{\gamma \lambda_p + \lambda_p - 1}), \quad \varepsilon \to 0.
    \end{align*}
\end{proof}

Finally, let $\Delta = \{0=t_0 < t_1<...<t_N=T\}$ be a uniform partition of $[0,T]$, $t_k = \frac{Tk}{N}$, $k=0,1,..., N$, $|\Delta|:=\frac{T}{N}$. For the given partition, we introduce
\begin{equation}\label{eq: taus and kappas}
\begin{gathered}
    \tau_-(t) := \max\{t_k,~t_k\le t\},
    \\
    \kappa_-(t) := \max\{k,~t_k\le t\},
    \\
    \tau_+(t) := \min\{t_k,~t_k\ge t\},
    \\
    \kappa_+(t) := \min\{k,~t_k \ge t\}.
\end{gathered}
\end{equation}
For any $n\ge n_0$, denote
\begin{equation}\label{EAdef}
    \hat Y^{N, n}_t := Y_0 + \int_0^t \tilde b_n \left(\tau_-(s), \hat Y^{N, n}_{\tau_-(s)}\right) ds + Z_{\tau_-(t)}.
\end{equation}
For each $n > n_0$, define also the function $y_n$: $[0,T] \to \mathcal D_0$ by
\[
    y_n(t) := \min\{ y> \varphi(t):~b(t,y) \le n \}
\]
and consider 
\begin{equation}\label{eq: deltan}
    \delta_n := \inf_{t\in[0,T]} (y_n(t) - \varphi(t)).
\end{equation}
\begin{remark}
    It is easy to see that $\varepsilon_n := \left(\frac{c}{n}\right)^{\frac{1}{\gamma}} \le \delta_n$. Moreover, $\delta_n \to 0$ as $n\to\infty$. Indeed, by the definition of $y_n$, for any fixed $t\in [0,T]$ and $n>n_0$
    \[
        y_n(t) \ge y_{n+1}(t)
    \]
    and hence $\delta_n \ge \delta_{n+1}$. Now, consider an arbitrary $\varepsilon \in(0,y_*)$ and take $n_\varepsilon := [\max_{t\in[0,T]} b(t, \varphi(t) + \varepsilon)]$ with $[\cdot]$ denoting the integer part. Then
    \[
        b(t, \varphi(t) + \varepsilon) < n_\varepsilon + 1
    \]
    for all $t\in[0,T]$. On the other hand, by assumption \textbf{(A3)},
    \[
        b(t, \varphi(t) + \varepsilon') \ge n_\varepsilon+1   
    \]
    for all $\varepsilon' < \left( \frac{c}{n_\varepsilon+1} \right)^{\frac{1}{\gamma}}$ which, by assumption \textbf{(A1)}, implies that for each $t\in[0,T]$
    \[
        y_{n_\varepsilon + 1}(t) - \varphi(t) < \varepsilon,
    \]
    i.e. $\delta_{n_\varepsilon + 1} < \varepsilon$. This, together with $\delta_n$ being decreasing, yields that $\delta_n \to 0$ as $n\to\infty$.
\end{remark}

\begin{theorem}\label{th: Euler scheme for Y}
    Let assumptions \textbf{(A1)}--\textbf{(A5)} hold and the noise $Z$ satisfying Assumptions \textbf{(Z1)}--\textbf{(Z2)} is such that
    \[
        \mathbb E \left[ |Z(t) - Z(s)|^p \right] \le C_{ \lambda, p}|t-s|^{\lambda p}, \quad s,t\in[0,T],
    \]
    for some $p\ge 2$ such that $\lambda_p := \lambda - \frac{2}{p} > \frac{1}{1+\gamma}$ with $\gamma$ from assumption \textbf{(A3)} and a positive constant $C_{\lambda, p} > 0$. Then
    \[
        \mathbb E\left[ \sup_{t\in[0,T]}\left|Y_t - \hat Y^{N, n}_t\right| \right] \le C \left(\delta_n^{\frac{\gamma\lambda_p + \lambda_p - 1}{2}} + \frac{(1+c_n)e^{ c_n}}{N^{\lambda_p}}\right),
    \]
    where $C$ is some positive constant that does not depend on $n$ or the mesh of the partition $|\Delta| = \frac{T}{N}$, $\delta_n$ is defined by \eqref{eq: deltan}, $\delta_n\to 0$, $n\to\infty$, and $c_n$ is from \eqref{eq: Lipschitz condition for approximation}.
\end{theorem}
\begin{proof}
    Just like in the proof of Proposition \ref{prop: probability of small Y}, observe that
    \[
        |Z_t-Z_s| \le \Lambda_p|t-s|^{\lambda_p},
    \]
    where
    \[
        \Lambda_p := A_{\lambda,p} \left(\int_0^T \int_0^T \frac{|Z_x - Z_y|^p}{|x-y|^{\lambda p }} dx dy\right)^{\frac{1}{p}},
    \]
    and note that the condition $p\ge 2$ implies that
    \[
        \mathbb E\Lambda_p^2 \le \left(\mathbb E \Lambda^p_p\right)^{\frac{2}{p}} < \infty.
    \]
    It is clear that
    \begin{align*}
        \mathbb E\left[ \sup_{t\in[0,T]}\left|Y_t - \hat Y^{N, n}_t\right| \right] & \le  \mathbb E\left[ \sup_{t\in[0,T]}\left|Y_t - \tilde Y^{(n)}_t\right| \right] +  \mathbb E\left[ \sup_{t\in[0,T]}\left|\tilde Y^{(n)}_t - \hat Y^{N, n}_t\right| \right].
    \end{align*}
    Let is estimate both terms in the right-hand side of the inequality above separately. Observe that
    \[
        b(t,y) = \tilde b_n(t,y), \quad (t,y)\in \mathcal D_{\delta_n}
    \]
    with $\delta_n$ defined by \eqref{eq: deltan}. Consider the set
    \[
        \mathcal A_n := \{\omega\in\Omega~|~\min_{t\in[0,T]}(Y_t(\omega)-\varphi(t)) > \delta_n\}
    \]
    and note that
    \[
        b(t,Y_t)\mathbbm 1_{\mathcal A_n} = b_n(t,Y_t)\mathbbm 1_{\mathcal A_n},
    \]
    i.e., for all $\omega\in \mathcal A_n$ the path $Y_t(\omega)$ satisfies the equation \eqref{eq: tilde Yn} and thus coincides with $\tilde Y^{(n)}_t(\omega)$. Whence
    \begin{align*}
        \mathbb E\left[ \sup_{t\in[0,T]}\left|Y_t - \tilde Y^{(n)}_t\right| \right] & = \mathbb E\left[ \sup_{t\in[0,T]}\left|Y_t - \tilde Y^{(n)}_t\right| \mathbbm 1_{\mathcal A_n}\right] + \mathbb E\left[ \sup_{t\in[0,T]}\left|Y_t - \tilde Y^{(n)}_t\right| \mathbbm 1_{\Omega \setminus \mathcal A_n}\right]
        \\
        &= \mathbb E\left[ \sup_{t\in[0,T]}\left|Y_t - \tilde Y^{(n)}_t\right| \mathbbm 1_{\Omega \setminus \mathcal A_n}\right]
        \\
        &\le \left(\mathbb E \left[ \left(\sup_{t\in[0,T]}\left|Y_t - \tilde Y^{(n)}_t\right|\right)^2 \right]\right)^{\frac{1}{2}} \sqrt{\mathbb P\left( \min_{t\in[0,T]}(Y_t - \varphi(t)) > \delta_n \right)}.
    \end{align*}
    By Theorem \ref{upper bound for Y} and Corollary \ref{cor: upper bound for approximations} applied w.r.t. $\lambda_p = \lambda - \frac{2}{p}$,
    \begin{align*}
        \mathbb E \left[ \left(\sup_{t\in[0,T]}\left|Y_t - \tilde Y^{(n)}_t\right|\right)^2 \right] & \le C \left(\mathbb E \left[ \left(\sup_{t\in[0,T]}\left|Y_t\right|\right)^2 \right] + \mathbb E \left[ \left(\sup_{t\in[0,T]}\left|\tilde Y^{(n)}_t\right|\right)^2 \right]\right)
        \\
        &\le C\left( 1+ \mathbb E\Lambda_p^2 \right) < \infty,
    \end{align*}
    and, by Proposition \ref{prop: probability of small Y}, there exists a constant $C>0$ such that
    \[
        \sqrt{\mathbb P\left( \min_{t\in[0,T]}(Y_t(\omega)-\varphi(t)) > \delta_n \right)} \le C\delta_n^{\frac{\gamma\lambda_p + \lambda_p - 1}{2}}.
    \]
    Therefore, there exists a constant $C>0$ that does not depend on $n$ or $N$ such that
    \begin{equation}\label{proofeq: approximation theorem est 1}
        \mathbb E\left[ \sup_{t\in[0,T]}\left|Y_t - \tilde Y^{(n)}_t\right| \right] \le C\delta_n^{\frac{\gamma\lambda_p + \lambda_p - 1}{2}}.
    \end{equation}
    
    Next, taking into account \eqref{eq: Lipschitz condition for approximation}, for any $t\in[0,T]$ we can write
    \begin{align*}
        \left|\tilde Y^{(n)}_t - \hat Y^{N, n}_t\right| & \le \int_0^t \left| \tilde b_n (s, \tilde Y^{(n)}_s) -  \tilde b_n \left(\tau_-(s), \tilde Y^{(n)}_s \right)\right| ds 
        \\
        &\qquad+ \int_0^t \left| \tilde b_n (\tau_-(s), \tilde Y^{(n)}_s) -  \tilde b_n \left(\tau_-(s), \hat Y^{N, n}_{\tau_-(s)}\right)\right| ds
        \\
        &\qquad + \Lambda_p |\Delta|^{\lambda_p}
        \\
        &\le c_n |\Delta|^{\lambda_p} + c_n \int_0^t \left|\tilde Y^{(n)}_s - \hat Y^{N, n}_s\right| ds + \Lambda_p |\Delta|^{\lambda_p},
    \end{align*}
    whence, since $\mathbb E\Lambda_p < \infty$,
    \[
        \mathbb E\left[ \sup_{s\in[0,t]} \left|\tilde Y^{(n)}_t - \hat Y^{N, n}_t\right| \right]  \le c_n |\Delta|^{\lambda_p} + c_n \int_0^t \mathbb E\left[\sup_{u\in[0,s]}\left|\tilde Y^{(n)}_u - \hat Y^{N, n}_u\right|\right] ds + C |\Delta|^{\lambda_p},
    \]
    and, by Gronwall's inequality, there exists constant $C>0$ such that
    \begin{equation*}
    \begin{aligned}
        \mathbb E\left[ \sup_{t\in[0,T]}\left|\tilde Y^{(n)}_t - \hat Y^{N, n}_t\right| \right] \le \frac{C(1+c_n)e^{c_n}}{N^{\lambda_p}}.
    \end{aligned}    
    \end{equation*}
    This, together with \eqref{proofeq: approximation theorem est 1}, finalizes the proof.
\end{proof}

\begin{remark}
    \begin{enumerate}
        \item As it is stated in \cite{ASVY2014}, the condition
        \[
            \mathbb E \left[ |Z(t) - Z(s)|^p \right] \le C_{ \lambda, p}|t-s|^{\lambda p}, \quad s,t\in[0,T],
        \]
        from Theorem \ref{th: Euler scheme for Y} is satisfied by any $\lambda$-H\"older continuous Gaussian process $Z$, i.e. processes from Examples \ref{ex: Gaussian Holder processes}--\ref{ex: mBm} fit into this framework.
        \item The process from Example \ref{ex: non-Gaussian continuous martingales} also satisfies conditions of Theorem \ref{th: Euler scheme for Y} as shown by \eqref{ex: non-Gaussian continuous martingales}.
    \end{enumerate}
     
\end{remark}

The sandwiched case presented in section \ref{sec: two-sided bounds} can be treated in the same manner. Instead of assumption \textbf{(A5)}, one should use the following one:
\begin{itemize}
    \item[\textbf{(B5)}] for any $\varepsilon_1, \varepsilon_2 > 0$, $\varepsilon_1 + \varepsilon_2 \le \lVert \varphi - \psi\rVert_\infty$, there is a constant $c_{\varepsilon_1, \varepsilon_2} > 0$ such that for any $(t,y)$, $(s,y) \in \mathcal D_{\varepsilon_1, \varepsilon_2}$:
    \[
    |b(t,y) - b(s, y)| \le c_{\varepsilon_1, \varepsilon_2} |t-s|^\lambda,
    \]
\end{itemize}
where $\mathcal D_{\varepsilon_1, \varepsilon_2}$ is defined by \eqref{eq: def of D}. Namely, let 
\[
    n_0 > \max\left\{\max_{t\in[0,T]}|b(t, \varphi(t) + y_*)|, \max_{t\in[0,T]}|b(t, \psi(t) - y_*)|\right\},
\]
where $y_*$ is from Assumption \textbf{(B3)}. For an arbitrary $n \ge n_0$ define
\begin{equation*}
\begin{gathered}
    \mathcal G^{\varphi}_n := \{(t,y) \in \mathcal D_{0,0}\setminus\mathcal D_{y_*,0}~|~b(t,y) < n \}
    \\
    \mathcal G^{\psi}_n := \{(t,y) \in \mathcal D_{0,0}\setminus\mathcal D_{0,y_*}~|~b(t,y) > -n \}
\end{gathered}
\end{equation*}
and consider the functions $\tilde b_n$: $[0,T]\times\mathbb R \to \mathbb R$ of the form
\begin{equation}\label{eq: bn tilde two-sided}
\tilde b_n(t, y) := \begin{cases}
    b(t, y), &\quad (t,y) \in \mathcal G^{\varphi}_n \cup \mathcal G^{\psi}_n \cup \mathcal D_{y_*, y_*},
    \\
    n, &\quad (t,y) \in [0,T]\times\mathbb R \setminus \left(\mathcal G^\varphi_n \cup \mathcal D_{y_*, -\infty} \right),
    \\
    -n, &\quad (t,y) \in [0,T]\times\mathbb R \setminus \left(\mathcal G^\psi_n \cup \mathcal D_{,-\infty, y_*} \right),
\end{cases}
\end{equation}
where 
$$\mathcal D_{y_*, -\infty} := \{(t,y)~|~t\in[0,T], y\in (\varphi(t)+y_*, \infty)\}$$
and 
$$\mathcal D_{ -\infty, y_*} := \{(t,y)~|~t\in[0,T], y\in (-\infty, \psi(t)-y_*)\}.$$
Denote $\varepsilon_n := \left(\frac{c}{n}\right)^{\frac{1}{\gamma}}$, and observe that 
\begin{equation}\label{eq: Lipschitz condition for approximation two-sided}
\begin{gathered}
    |\tilde b_n(t,y_1) - \tilde b_n(t,y_2)| \le c_n|y_1-y_2|, \quad t\in[0,T], \quad y_1,y_2\in \mathbb R,
    \\
    |\tilde b_n(t_1,y) - \tilde b_n(t_2,y)| \le c_n|t_1-t_2|^\lambda, \quad t_1,t_2\in[0,T], \quad y\in \mathbb R,
\end{gathered}
\end{equation}
where $c_n := c_{\varepsilon_n, \varepsilon_n}$ denotes the constant from assumptions \textbf{(B2)} and \textbf{(B5)} which corresponds to $\varepsilon_n = \left(\frac{c}{n}\right)^{\frac{1}{\gamma}}$. In particular, this implies that the SDE
\begin{equation}\label{eq: tilde Yn two-sided}
    d\tilde Y^{(n)}_t = \tilde b_n(t,\tilde Y^{(n)}_t)dt + d Z_t, \quad \tilde Y^{(n)}_0 = Y_0 > \varphi(0),
\end{equation}
has a unique pathwise solution and, just like in the one-sided case, can be simulated via the standard Euler scheme:
\begin{equation}\label{EAdef two-sided}
    \hat Y^{N, n}_t := Y_0 + \int_0^t \tilde b_n \left(\tau_-(s), \hat Y^{N, n}_{\tau_-(s)}\right) ds + Z_{\tau_-(t)}.
\end{equation}

Now, for each $n\ge n_0$ denote
\[
    y^\varphi_n(t) := \min\{ y\in (\varphi(t), \psi(t)):~b(t,y) \le n \}, \quad y^\psi_n(t) := \max\{ y \in (\varphi(t), \psi(t)):~b(t,y) \ge -n \}
\]
and define
\begin{equation}\label{eq: deltan two-sided}
    \delta_n := \max\left\{\inf_{t\in[0,T]} (y^\varphi_n(t) - \varphi(t)), \inf_{t\in[0,T]} (\psi(t) - y^\psi_n(t))\right\}.
\end{equation}
Similarly to the one-seded case, it is straightforward to prove that $\delta_n \to 0$, $n\to \infty$.

Now we are ready to formulate the two-sided counterpart of Theorem \ref{th: Euler scheme for Y}.

\begin{theorem}\label{th: Euler scheme for Y two-sided}
    Let assumptions \textbf{(B1)}--\textbf{(B5)} hold and the noise $Z$ satisfying Assumptions \textbf{(Z1)}--\textbf{(Z2)} is such that
    \[
        \mathbb E \left[ |Z(t) - Z(s)|^p \right] \le C_{ \lambda, p}|t-s|^{\lambda p}, \quad s,t\in[0,T],
    \]
    for some $p\ge 2$ such that $\lambda_p := \lambda - \frac{2}{p} > \frac{1}{1+\gamma}$ with $\gamma$ from assumption \textbf{(B3)} and a positive constant $C_{\lambda, p} > 0$. Then
    \[
        \mathbb E\left[ \sup_{t\in[0,T]}\left|Y_t - \hat Y^{N, n}_t\right| \right] \le C \left(\delta_n^{\frac{\gamma\lambda_p + \lambda_p - 1}{2}} + \frac{(1+c_n)e^{ c_n}}{N^{\lambda_p}}\right),
    \]
    where $C$ is some positive constant that does not depend on $n$ or the mesh of the partition $|\Delta| = \frac{T}{N}$, $\delta_n$ is defined by \eqref{eq: deltan two-sided}, $\delta_n\to 0$, $n\to\infty$, and $c_n$ is from \eqref{eq: Lipschitz condition for approximation two-sided}.    
\end{theorem}

\begin{remark}
    Theorems \ref{th: Euler scheme for Y} and \ref{th: Euler scheme for Y two-sided} guarantee convergence for all $\lambda \in (0,1)$, but in practice the scheme performs much better for $\lambda$ close to 1. The reason is as follows: in order to make $\delta_n^{\frac{\gamma\lambda_p + \lambda_p - 1}{2}}$ small, one has to consider large values of $n$; this results in larger values of $(1+c_n)e^{ c_n}$ that, in turn, have to be ``compensated'' by the denominator $N^{\lambda_p}$. The bigger is $\lambda_p$, the smaller values of $n$ (and hence of $N$) can be.  
\end{remark}

\subsection{Simulations}

To conclude the work, we illustrate the results presented in this paper using the semi-heuristic Euler approximation scheme considered previously. Note that the scheme does not guarantee that the discretized process remains between $\varphi$ and $\psi$, but in practice the property of being sandwiched is not violated to a big extent (see below). All the simulations are performed in the \textsf{R} programming language on the system with Intel Core i9-9900K CPU and 64 Gb RAM. In order to simulate paths of fractional Brownian motion, \textsf{R} package \textsf{somebm} is used.

\subsubsection{Simulation 1: square root of fractional Cox--Ingersoll--Ross process}

As the first example, consider a particular example of the process described in subsection \ref{Ex: CIR-CEV}, namely the square root of the fractional Cox--Ingersoll-Ross process:
\begin{equation}\label{eq: CIR simulations}
     Y_t = Y_0 + \frac{1}{2} \int_0^t \left(\frac{\kappa}{Y_s} - \theta Y_s\right)ds + \frac{\sigma}{2} B^H_t,\quad t\in[0,T]
\end{equation}
where $Y_0$, $\kappa$, $\theta$ and $\sigma$ are positive constants and $B^H$ is a fractional Brownian motion with Hurst index $H>\frac{1}{2}$. Note that such process is a convenient subject for testing our approximation scheme since \cite[Theorem 4.1]{Hong2020} gives an alternative backward Euler approximation method for it and one can use that algorithm as a reference for performance evaluation. 

In our simulations, we take $N = 100000$, $T=1$, $Y_0 = 1$, $\kappa = 3$, $\theta =1$, $\sigma = 1$, $H = 0.7$ (these values of parameters satisfy conditions of \cite[Theorem 4.1]{Hong2020}). For the semi-heuristic Euler scheme, we take $n=20$. We generated 1000 paths of fractional Brownian motion and used them to simulate the process \eqref{eq: CIR simulations} with the semi-heuristic Euler approximation method (10 trajectories are given on Fig. \ref{fig: CIR using SEE scheme}) and with the algorithm from \cite{Hong2020}. On average, simulation time of one path using the semi-heuristic Euler scheme was slightly slower in comparison to the backward Euler scheme (0.08577108 seconds versus 0.03191614 seconds).
\begin{figure}[h!]
    \centering
    \includegraphics[width=0.75\textwidth]{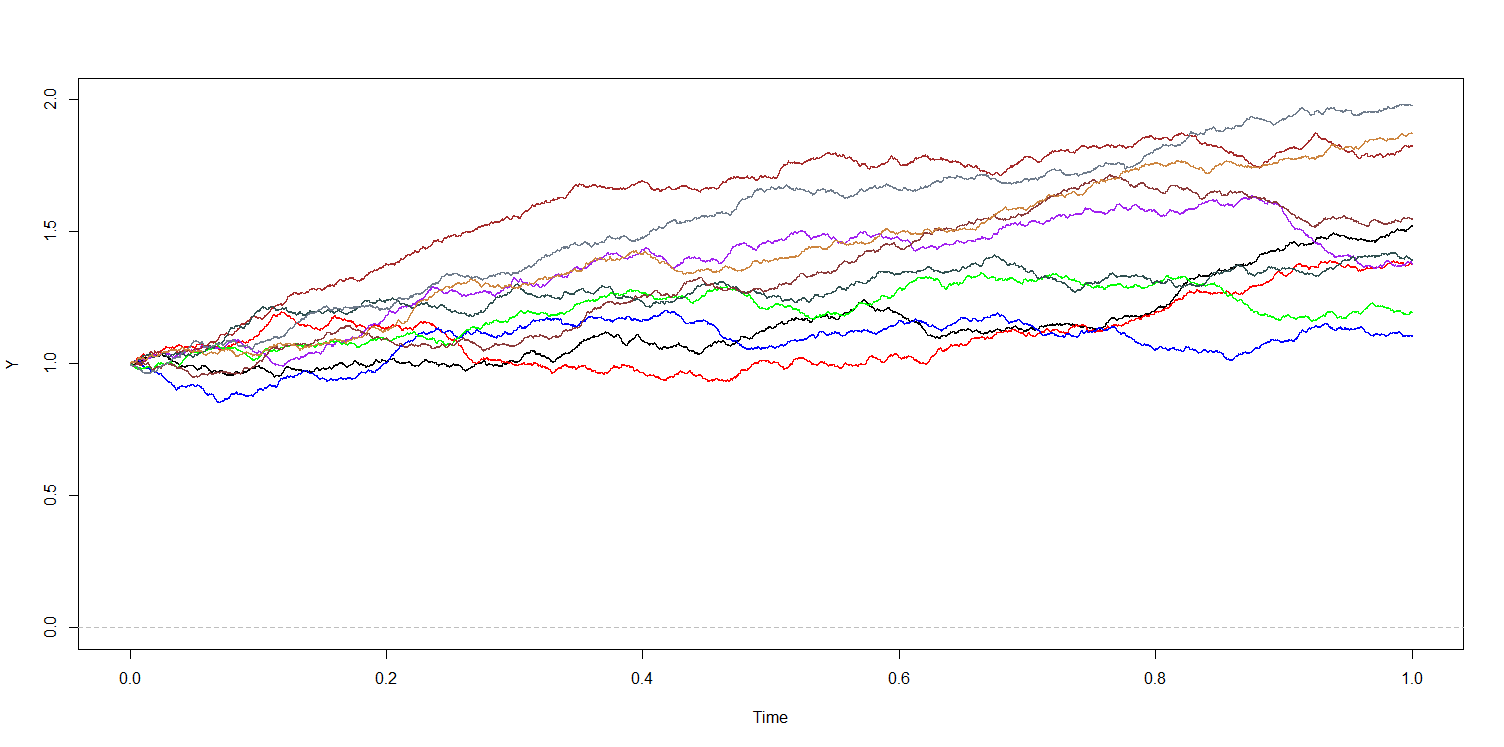}
    \caption{Ten sample paths of \eqref{eq: CIR simulations} generated using the semi-heuristic Euler approximation scheme; $N = 100000$, $T=1$, $Y_0 = 1$, $\kappa = 3$, $\theta =1$, $\sigma = 1$, $H = 0.7$, $n=20$. }\label{fig: CIR using SEE scheme}
\end{figure}

Afterwards, we calculated
\begin{equation}\label{eq: comparison of schemes}
    \sup_{k=0,...,N} |\hat Y^{N,n}_{t_k} - \hat Y_{t_k}|
\end{equation}
for the corresponding generated paths ($\hat Y$ above denotes the backward Euler scheme from \cite{Hong2020}). The boxplot of values \eqref{eq: comparison of schemes} based on 1000 simulations is given on Fig. \ref{fig: CIR comparison}(a). As we can see, both schemes give almost identical results with the value \eqref{eq: comparison of schemes} taking values roughly between $10^{-6}$ and $10^{-5}$. In fact, paths generated with these two methods are indistinguishable on the plots (see Fig.~\ref{fig: CIR comparison}(b)) despite worse convergence rate of the semi-heuristic scheme provided by Theorem \ref{th: Euler scheme for Y}. Moreover, all 1000 paths of \eqref{eq: CIR simulations} simulated with the semi-heuristic method stay above 0.

\begin{figure}[h!]
  \centering
\begin{minipage}[b]{0.4\textwidth} \centering
    \includegraphics[width=\textwidth]{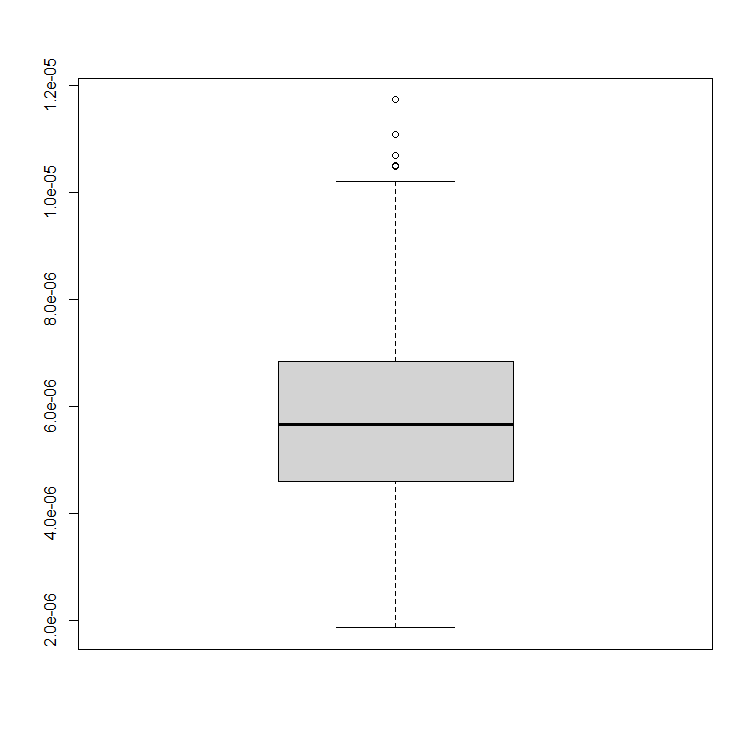}
(a) 
\end{minipage}
\begin{minipage}[b]{0.4\textwidth} \centering
\includegraphics[width=\textwidth]{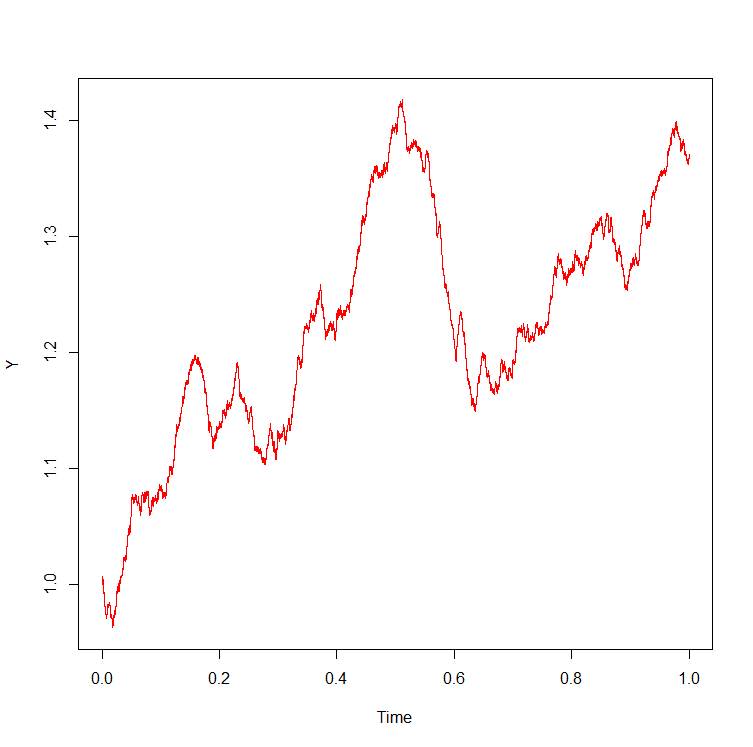}
(b) 
\end{minipage}
\caption{ Comparison of two simulation methods. (a) Boxplot of \eqref{eq: comparison of schemes} based on 1000 simulations. (b) Path of \eqref{eq: CIR simulations} generated by the semi-heuristic (black) and backward (red) Euler methods. Note that both trajectories completely coincide and therefore the black path is not visible on the picture. }\label{fig: CIR comparison}
\end{figure}

\subsubsection{Simulation 2: two-sided sandwiched process with equidistant bounds}

As the second example, we take
\begin{equation}\label{eq: simulations 1}
    Y_t = 2.5 + \int_0^t \left(\frac{1}{(Y_s - \cos(5s))^4} - \frac{1}{(3+ \cos(5s) - Y_s)^4}\right)ds + 3 B^H_t,\quad t\in[0,1],
\end{equation}
with
\[
    \psi(t) - \varphi(t) = 3+ \cos(5t) - \cos(5t) = 3, \quad t\in[0,1].
\]
In all simulations, $N = 100000$, $n=20$. Simulation of one trajectory takes approximately 0.1601701 seconds. All 1000 generated paths (10 of simulated paths are presented on Fig.~\ref{fig: sim2}) stay between the bounds without crossing them.

\begin{figure}[h!]
    \centering
    \includegraphics[width=0.7\textwidth]{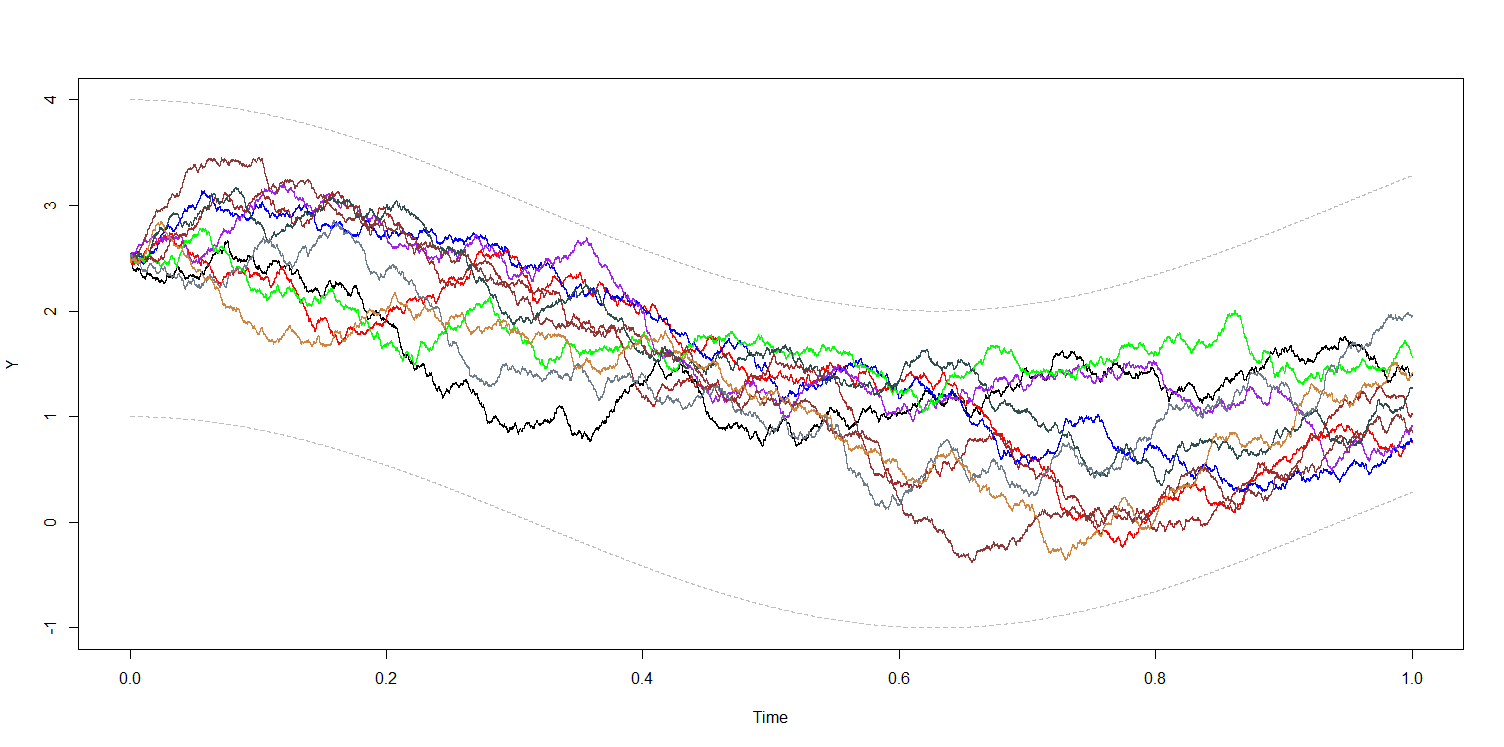}
    \caption{Ten sample paths of \eqref{eq: simulations 1} generated using the semi-heuristic Euler approximation scheme, $N = 100000$, $n=20$.}\label{fig: sim2}
\end{figure}

\subsubsection{Simulation 3: two-sided sandwiched process with shrinking bounds}

As the final example, we consider
\begin{equation}\label{eq: simulations 2}
    Y_t = \int_0^t \left(\frac{1}{(Y_s + e^{-s})^4} - \frac{1}{(e^{-s} - Y_s)^4}\right)ds + B^H_t,\quad t\in[0,1],
\end{equation}
with
\[
    \psi(t) - \varphi(t) = 2e^{-t} \to 0, \quad t\to\infty.
\]
Once again, $N = 100000$, $n=20$. Simulation of one trajectory takes approximately 0.1575801 seconds and all 1000 generated paths (10 of simulated paths are presented on Fig. \ref{fig: sim3}) stay between the bounds without crossing them.

\begin{figure}[h!]
    \centering
    \includegraphics[width=0.75\textwidth]{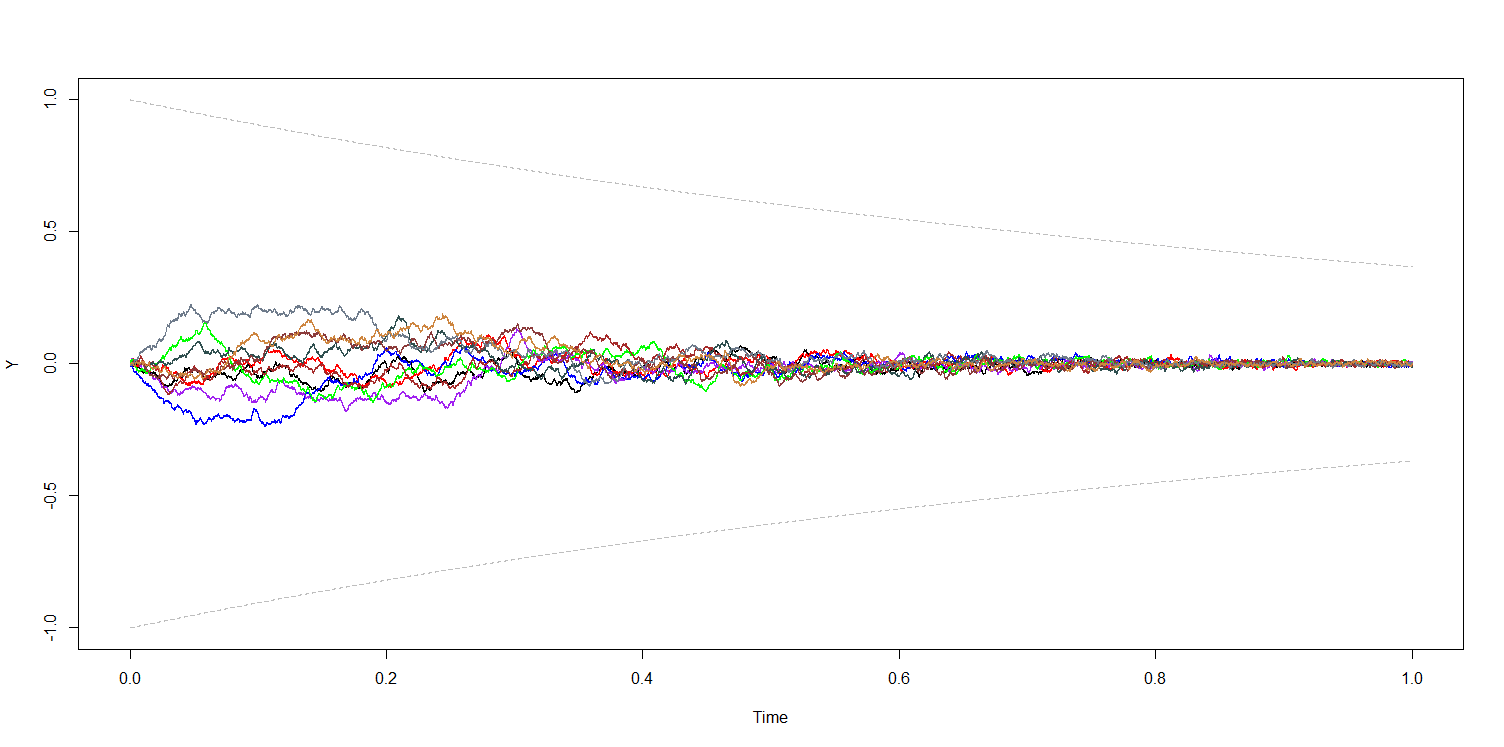}
    \caption{Ten sample paths of \eqref{eq: simulations 2} generated using the semi-heuristic Euler approximation scheme, $N = 100000$, $n=20$.}\label{fig: sim3}
\end{figure}

\end{section}

\section*{Acknowledgements}

The present research is carried out within the frame and support of the ToppForsk project nr. 274410 of the Research Council of Norway with title STORM: Stochastics for Time-Space Risk Models. The second author is supported  by  the Ukrainian research project "Exact formulae, estimates, asymptotic properties and statistical analysis of complex evolutionary systems with many degrees of freedom" (state registration number 0119U100317).  

\bibliographystyle{acm}
\bibliography{biblio.bib}


\appendix
\newpage

\section{Appendix: Existence of the local solution}\label{sec: existence of local solution}

In this Appendix, we give a proof of Theorem \ref{thm: existence of the local solution} on the existence of the solution to \eqref{volatility1} under assumptions \textbf{(A1)}--\textbf{(A3)} until the first moment of hitting $\varphi$ by the latter. Note that it would be possible to prove this result using a modification of the standard Picard iteration argument, but we choose a different strategy: we approximate the non-Lipschitz drift of \eqref{volatility1} by a sequence of the Lipschitz ones, obtain a monotonically increasing sequence of processes and prove that their limit is the only solution. Choice of such a method is explained by two points. First, without assumption \textbf{(A4)}, the solution may hit $\varphi$ and the limiting procedure described in this Appendix allows to see (up to some extent) what happens beyond this moment. Second, the pre-limit processes are very easy to simulate, so they can be used for numerical schemes.

Before going to the proof of Theorem \ref{thm: existence of the local solution}, we will require several auxiliary results. Let $n_0 > \max_{t\in[0,T]}|b(t, \varphi(t) + y_*)|$, where $y_*$ is from Assumption \textbf{(A3)}. For an arbitrary $n \ge n_0$ define the set
\begin{equation*}
\begin{gathered}
\mathcal G_n := \{(t,y) \in \mathcal D_{0}\setminus\mathcal D_{y_*}~|~b(t,y) < n \}
\end{gathered}
\end{equation*}
and consider the functions $\tilde b_n$: $[0,T]\times\mathbb R \to \mathbb R$ of the form
\begin{equation}\label{bn tilde}
\tilde b_n(t, y) := \begin{cases}
b(t, y), &\quad (t,y) \in \mathcal G_n \cup \mathcal D_{y_*},
\\
n, &\quad (t,y) \in [0,T]\times\mathbb R \setminus \left(\mathcal G_n \cup \mathcal D_{y_*} \right),
\end{cases}
\end{equation}
$b_n(t, y) := \tilde b_n(t,y) - \frac{1}{n}$.

Note that each $b_n$ is Lipschitz continuous, i.e. for all $(t,y_1), (t, y_2) \in [0,T]\times \mathbb R$ there exists the constant $C$ that depends on $n$ but does not depend on $t$ such that
\[
|b_n(t, y_1) - b_n(t, y_2)| \le C |y_1 - y_2|.
\]
Using this fact, it is straightforward to prove by the standard fixed point argument that the stochastic differential equation of the form
\begin{equation}\label{Yn}
    dY^{(n)}_t = b_n(t, Y^{(n)}_t)dt + dZ_t, \quad Y^{(n)}_0 = Y_0>0,
\end{equation}
has a pathwisely unique solution. 

In order to progress, we will require a simple comparison-type result.

\begin{lemma}
Assume that continuous random processes $\{X_1 (t),~t\ge 0\}$ and $\{X_2(t), t\ge 0\}$ satisfy (a.s.) the equations of the form
\[
X_i(t) = X_0 + \int_0^t f_i (s, X_i(s)) ds + Z_t, \quad t\ge 0, \quad i = 1,2,
\]
where $X_0$ is a constant and $f_1$, $f_2$: $[0,\infty)\times \mathbb R \to \mathbb R$ are continuous functions such that for any $(t,x) \in [0,\infty)\times \mathbb R$:
\[
f_1 (t, x) < f_2 (t,x).
\]
Then $X_1(t) < X_2(t)$ a.s. for any $t > 0$.
\end{lemma}

\begin{proof}
The proof is straightforward. Denote
\[
\Delta(t) := X_2(t) - X_1(t) = \int_0^t \left(f_2(s, X_2(s)) - f_1(s, X_1(s))\right)ds, \quad t\ge 0,
\]
and observe that $\Delta(0) = 0$ and that the function $\Delta$ is differentiable with
\[
\Delta '_+ (0) = f_2(0, X_0) - f_1(0, X_0) > 0.
\]
It is clear that $\Delta(t) = \Delta '_+ (0) t + o(t)$, $t\to 0+$, whence there exists the maximal interval $(0, t^*) \subset (0,\infty)$ such that $\Delta(t) > 0$ for all $t\in(0,t^*)$. It is also clear that
\[
t^* = \sup\{t>0~|~\forall s \in (0,t): \Delta(s) >0 \}.
\]
Assume that $t^* < \infty$. By the definition of $t^*$ and continuity of $\Delta$, $\Delta(t^*) = 0$. Hence $X_1(t^*) = X_2(t^*) = X^*$ and
\[
\Delta ' (t^*) = f_2(t^*, X^*) - f_1(t^*, X^*) > 0.
\]
As $\Delta (t) = \Delta'(t^*) (t- t^*) + o(t-t^*)$, $t \to t^*$, there exists such $\varepsilon >0$ that $\Delta(t)<0$ for all $t \in (t^* - \varepsilon, t^*)$ which contradicts the definition of $t^*$. Therefore $t^* = \infty$ and for all $t>0$:
\[
X_1(t) < X_2(t).
\]
\end{proof}

It is easy to observe that $b_n(t, y) < b_{n+1}(t,y)$ for any $n\ge 1$ and $(t,y) \in [0,T]\times \mathbb R$, whence $Y^{(n)}_t < Y^{(n+1)}_t$ for all $t\in(0,T]$ and therefore one can define a limit $Y^{(\infty)}_t := \lim_{n\to\infty} Y^{(n)}_t \in (-\infty, \infty]$, $t\in[0,T]$.

\begin{proposition}\label{finiteness of Y infty}
Let assumptions \textbf{(A1)}--\textbf{(A3)} hold. Then, there is a random variable $\Psi >\max_{t\in[0,T]}|\varphi(t)|$ such that for any $t\in [0,T]$:
\[
|Y^{(\infty)}_t|  \le \Psi <\infty.
\]
\end{proposition}

\begin{proof}

Denote $\eta := \frac{Y_0 - \varphi(0)}{2}$ and consider
\begin{align*}
  \tau^n_1 :&= \sup\left\{s\in [0,T]~\bigg|~\forall u \in [0,s]:~Y^{(n)}_u \ge \varphi(u) + \eta \right\}
\\
&= \inf\left\{s\in[0,T]~\bigg|~Y^{(n)}_s < \varphi(s) + \eta \right\} \wedge T.
\end{align*}
We shall first prove that for all $n\ge n_0$:
\begin{align*}
|Y^{(n)}_t| & \le  |Y_0| + 2\max_{s\in[0,T]}|Y^{(n_0)}_s| + 5\max_{s\in[0,T]}|\varphi(s)| + \eta
\\
&\quad+ C t + C \int_{0}^t |Y^{(n)}_s| ds + 2\max_{s\in[0,T]}|Z_s|,
\end{align*}
with $C>0$ being a constant that does not depend on $n$. Then the required result follows by Gronwall's inequality. 

For the reader's convenience, we will divide the proof into several steps to separate cases $t\in[0,\tau^n_1]$ and $t\in (\tau^n_1, T]$. 
\\
\textbf{Step 1.} Fix an arbitrary $n\ge n_0$ and assume that $t \in [0,\tau^n_1]$, i.e. $Y^{(n)}_s \ge \varphi(s) + \eta$ for each $s\le t$. 
Observe that for all $(s,y) \in \mathcal D_{\eta}$
\begin{equation}\label{eq: linear growth of bn for upper bound}
|b_n(s,y)| \le C(1 + |y|),
\end{equation}
where $C>0$ is some constant that depends neither on $n$ nor on $s$. Indeed, it is easy to verify using definition of $b_n$ and assumption \textbf{(A3)} that for all  $(s,y) \in \mathcal D_{\eta}$
\[
|b_n(s,y)| \le \left|b_n(s,y) + \frac{1}{n}\right| + 1 \le |b(s,y)| + 1.
\]
Furthermore, by assumption \textbf{(A2)},  for all $(s,y) \in \mathcal D_{\eta}$
\begin{align*}
|b(s,y)| &\le \left|b(s, y) - b\left(s, \varphi(s) + \eta \right) \right| + \left| b\left(s, \varphi(s) + \eta\right)\right| 
\\
&\le c_\eta (y - \varphi(s) - \eta) + \max_{s\in[0,T]} \left| b\left(s, \varphi(s) + \eta\right)\right| 
\\
&\le \left(\max_{s\in[0,T]} |b\left(s, \varphi(s) + \eta\right)| + c_\eta \max_{s\in[0,T]}|\varphi(s)| + c_\eta (\eta + 1)   \right) (1+ |y|).
\end{align*}
Using \eqref{eq: linear growth of bn for upper bound}, for an arbitrary $n\ge n_0$:
 \begin{equation}\label{FM1}
  \begin{aligned}
   \left|Y^{(n)}_t\right| & = \left|Y_0 + \int_0^t b_n(s, Y^{(n)}_s) ds + Z_t \right|
   \\
   & \le |Y_0| + \int_0^t |b_n(s, Y^{(n)}_s)| ds + |Z_t| 
   \\
   & \le |Y_0| + Ct + C \int_0^t |Y^{(n)}_s|  ds+ \max_{s\in[0,T]}|Z_s|.
  \end{aligned}
 \end{equation}

\noindent \textbf{Step 2.} Assume $t > \tau^n_1$. Consider
\[
\tau_2^n(t) := \sup\left\{ s \in (\tau^n_1, t]~\bigg|~ |Y^{(n)}_s - \varphi(s)| <  \eta \right\}.
\]

Note that $\left|Y^{(n)}_{{\tau^n_2(t)}}\right| \le |\varphi\left(\tau_2^n(t)\right)|+\eta  \le \max_{s\in[0,T]} |\varphi(s)| + \eta$ and, therefore,
\begin{equation*}
\begin{aligned}
|Y^{(n)}_t| &= \left| \left(Y^{(n)}_t - Y^{(n)}_{\tau^n_2(t)}\right) + Y^{(n)}_{\tau^n_2(t)} \right| 
\\
&\le \left|Y^{(n)}_t - Y^{(n)}_{\tau^n_2(t)}\right| + \max_{s\in[0,T]} |\varphi(s)| + \eta.
\end{aligned}
\end{equation*}
If $\tau^n_2(t) = t$, then $| Y^{(n)}_t - Y^{(n)}_{\tau^n_2(t)} | = 0$, so $|Y^{(n)}_t| < \max_{s\in[0,T]} |\varphi(s)| + \eta$.
Otherwise, if $\tau^n_2(t) < t$, then, for any $s\in[\tau^n_2(t),t]$: $|Y^{(n)}_s - \varphi(s)| \ge  \eta$ which means that either $Y^{(n)}_s \le \varphi(s) - \eta$ or $Y^{(n)}_s \ge \varphi(s) + \eta$ for all $s\in[\tau^n_2(t),t]$.
In the first case, taking into account the monotonicity of $Y^{(n)}_s$ with respect to $n$, we have
\[
Y^{(n_0)}_s - \varphi(s) \le Y^{(n)}_s - \varphi(s) \le -\eta,
\]
i.e. 
\[
\eta \le |Y^{(n)}_s - \varphi(s)| \le |Y^{(n_0)}_s - \varphi(s)|,
\]
so
\begin{equation}\label{eq: negative case}
\begin{aligned}
| Y^{(n)}_t - Y^{(n)}_{\tau^n_2(t)}| &\le |Y^{(n)}_t - \varphi(t)| + |Y^{(n)}_{\tau^n_2(t)} - \varphi({\tau^n_2(t)})| +|\varphi(t) - \varphi(\tau^n_2(t))|
\\
&\le |Y^{(n_0)}_t - \varphi(t)| + |Y^{(n_0)}_{\tau^n_2(t)} - \varphi({\tau^n_2(t)})|+|\varphi(t) - \varphi(\tau^n_2(t))|
\\
&\le  2\max_{s\in[0,T]}|Y^{(n_0)}_t| + 4\max_{s\in[0,T]}|\varphi(s)|.
\end{aligned}
\end{equation}
In the second case, since $(s,  Y^{(n)}_s) \in \mathcal D_\eta$, we can use \eqref{eq: linear growth of bn for upper bound} to obtain that
\begin{equation*}
\begin{aligned}
| Y^{(n)}_t - Y^{(n)}_{\tau^n_2(t)} | &= \bigg| \int_{\tau^n_2(t)}^t b(s, Y^{(n)}_s) ds + (Z_t - Z_{\tau^n_2(t)})  \bigg|
\\
&\le  C (t - \tau^n_2(t)) + C \int_{\tau^n_2(t)}^t |Y^{(n)}_s| ds + 2\max_{s\in[0,T]}|Z_s|
\\
&\le C t + C \int_{0}^t |Y^{(n)}_s| ds + 2\max_{s\in[0,T]}|Z_s|.
\end{aligned}
\end{equation*}
In any situation, for all $t > \tau^n_1$:
\begin{equation}\label{FM2}
\begin{aligned}
|Y^{(n)}_t| & \le   2\max_{s\in[0,T]}|Y^{(n_0)}_s| + 5\max_{s\in[0,T]}|\varphi(s)| + \eta
\\
&\quad+ C t + C \int_{0}^t |Y^{(n)}_s| ds + 2\max_{s\in[0,T]}|Z_s|.
\end{aligned}
\end{equation}

\noindent \textbf{Step 3.} Taking into account \eqref{FM1} and \eqref{FM2}, it is easy to see that for all $t \ge 0$:
\begin{align*}
|Y^{(n)}_t| & \le |Y_0| + 2\max_{s\in[0,T]}|Y^{(n_0)}_s| + 5\max_{s\in[0,T]}|\varphi(s)| + \eta
\\
&\quad+ C t + C \int_{0}^t |Y^{(n)}_s| ds + 2\max_{s\in[0,T]}|Z_s|,
\end{align*}
so, by Gronwall's inequality, for all $n\ge 1$:
\begin{align}\label{eq: finiteness of Y}
|Y^{(n)}_t| & \le \Psi <\infty,
\end{align}
where 
\[
\Psi:= \left(|Y_0| + 2\max_{s\in[0,T]}|Y^{(n_0)}_s| + 5\max_{s\in[0,T]}|\varphi(s)| + \eta + C T + 2\max_{s\in[0,T]}|Z_s|\right) e^{C T}.
\]
Since the right-hand side of \eqref{eq: finiteness of Y} does not depend on $n$, the claim of the proposition holds for $Y^{(\infty)}$.
\end{proof}

\begin{proposition}\label{prop: positivity of Yinfty}
For all $t\in [0,T]$: $Y^{(\infty)}_t \ge \varphi(t)$.
\end{proposition}

\begin{proof}
\textbf{Step 1}. Fix an arbitrary $t\in[0,T]$ and denote
\[
b_n^+(s, y) := b_n(s, y) \vee 0, \qquad b_n^-(s,y) := -(b_n(s,y)\wedge 0),
\]
$ b_n(s, y) = b_n^+(s, y) - b_n^-(s, y)$.
\\
Observe that, by assumption \textbf{(A3)}, $b_n^-(s,y) = 0$ for all $(s,y) \in \mathcal D_0 \setminus \mathcal D_{y_*}$, and, by assumption \textbf{(A2)}, $b_n^-$ is globally Lipschitz continuous. From Proposition \ref{finiteness of Y infty} we obtain that, for some constant $L>0$ that does not depend on $n$ and for all $s\in[0,t]$:
\begin{align*}
|b_n^-(s, Y^{(n)}_s)| &\le L(1+|Y^{(n)}_s|) \le L(1 + \Psi) = : \widetilde\Psi,
\end{align*}
where $\widetilde\Psi$ is a finite random variable. Hence, since $b^-_n (s, Y^{(n)}_s) \to b^- (s, Y^{(\infty)}_s)$ pointwise as $n\to\infty$, by the dominated convergence theorem,
\[
\int_0^t b^-_n(s, Y^{(n)}_s)ds \to \int_0^t b^-(s, Y^{(\infty)}_s)ds, \quad n\to\infty.
\] 
Taking into account the convergence above and Proposition \ref{finiteness of Y infty}, the left hand side of
\[
Y^{(n)}_t - Y_0 - Z_t + \int_0^t b^-_n(s, Y^{(n)}_s)ds = \int_0^t b^+_n(s, Y^{(n)}_s)ds
\]
converges to a finite value as $n\to\infty$ for each $t\in [0,T]$. Therefore there exists the limit
\begin{equation}\label{eq: b plus integrable}
\lim_{n\to\infty} \int_0^t b^+_n(s, Y^{(n)}_s)ds < \infty.
\end{equation}
\noindent\textbf{Step 2.} Let us now prove that 
\[
\mu\{s\in[0,T]~|~Y^{(n)}_s \le \varphi(s)\} \to 0,\quad n\to\infty,
\]
 with $\mu$ being the Lebesgue measure on $[0,T]$. Assume that it is not true. i.e. there exist $\epsilon > 0$ and a subsequence $\{n_k:~k \ge 1\}$ such that for all $k\ge 1$:
\[
\mu\{s\in[0,T]~|~Y^{(n_k)}_s \le \varphi(s)\} \ge \epsilon > 0.
\]
In this case,
\begin{align*}
\int_0^T b^+_{n_k}(s, Y^{(n_k)}_s) ds &= \int_{\{s\in[0,T]~|~Y^{(n_k)}_s > \varphi(s) \}} b^+_{n_k}(s, Y^{(n_k)}_s) ds 
\\
&\quad+ \int_{\{s\in[0,T]~|~Y^{(n_k)}_s \le \varphi(s) \}} b^+_{n_k}(s, Y^{(n_k)}_s) ds
\\
&\ge \int_{\{s\in[0,T]~|~Y^{(n_k)}_s \le \varphi(s) \}} b^+_{n_k}(s, Y^{(n_k)}_s) ds
\\
&= \int_{\{s\in[0,T]~|~Y^{(n_k)}_s \le \varphi(s) \}} \left(n_k - \frac{1}{n_k}\right) ds
\\
&\ge n_k \epsilon - \frac{\epsilon}{n_k} \to\infty, \quad k\to\infty,
\end{align*}
that contradicts \eqref{eq: b plus integrable}.
\\
This implies that $\mu\{s\in[0,T]~|~Y^{(\infty)}_s \le \varphi(s) \} = 0$, i.e. $Y^{(\infty)}$ exceeds $\varphi$ a.e. on $[0,T]$.
\\
\textbf{Step 3}. Assume that there is such $\tau \in (0,T]$ that $Y^{(\infty)}_\tau < \varphi(\tau)$. Then, for all $n\ge 1$:
\[
Y^{(n)}_\tau < Y^{(\infty)}_\tau \le \varphi(\tau).
\]
Fix an arbitrary $n\ge n_0$ and denote 
\begin{equation*}
\begin{gathered}
\tau_-^n := \sup\{t\in[0,\tau)~|~Y^{(n)}_t > \varphi (t)\}.
\end{gathered}
\end{equation*}
Note that, due to continuity of $Y^{(n)}$ and Step 2, $0 < \tau_-^n < \tau \le T$. Furthermore, $Y^{(n)}_{\tau_-^n} - \varphi(\tau_-^n) = 0$ and for all $t \in (\tau_-^n, \tau]$: $Y^{(n)}_t \le \varphi(t)$. Next, for an arbitrary $t \in (\tau_-^n, \tau)$:
\begin{align*}
\varphi(t)\ge & Y^{(n)}_t = Y^{(n)}_t - Y^{(n)}_{\tau_-^n} + \varphi(\tau_-^n)
\\
=& \varphi(\tau_-^n) + \int_{\tau_-^n}^t b_{n}(s, Y^{(n)}_s)ds + (Z_t - Z_{\tau_-^n})
\\
= & \varphi(\tau_-^n) + \left(n - \frac{1}{n}\right)(t - \tau_-^n) + (Z_t - Z_{\tau_-^n})
\\
\ge& \varphi(\tau_-^n) + \left(n - \frac{1}{n}\right)(t - \tau_-^n) - \Lambda (t - \tau_-^n)^\lambda,
\end{align*}
therefore, for any $n\ge n_0$:
\begin{equation}\label{eq: if Yinfty is negative}
0 > Y^{(\infty)}_\tau - \varphi(\tau) > Y^{(n)}_{\tau} - \varphi(\tau) \ge \varphi(\tau_-^n) - \varphi(\tau) + \min_{t\in[\tau_-^n, \tau]} F^{(n)}(t),
\end{equation}
with $F^{(n)}(t) := \left(n - \frac{1}{n}\right) (t - \tau_-^n) - \Lambda (t - \tau_-^n)^\lambda$. However, $\min_{t\in[\tau_-^n, \tau]} F^{(n)}(t) \to 0$, $n\to\infty$. Indeed, 
\[
\min_{t\in[0,\infty)} F^{(n)}(t) \le \min_{t\in[\tau_-^n, \tau]} F^{(n)}(t) \le 0
\]
and it is straightforward to verify that $F^{(n)}(t)$ takes its minimal value on $[0,\infty)$ at 
\[
t_* := \tau_-^n + \left( \frac{\lambda \Lambda}{n - \frac{1}{n}} \right)^{\frac{1}{1 - \lambda}} 
\]
with
\[
F^{(n)}(t_*) = \frac{\Lambda^{\frac{1}{1-\lambda}} (\lambda^{\frac{1}{1-\lambda}} - \lambda^{\frac{\lambda}{1-\lambda}})}{\left(n - \frac{1}{n}\right)^{\frac{\lambda}{1-\lambda}}} \to 0, \quad n\to\infty.
\]
Furthermore, it is easy to see from Step 2 that $\tau^n_- \to \tau$, $n\to\infty$, so $ \varphi(\tau_-^n) - \varphi(\tau) \to 0$, $n\to\infty$, and therefore \eqref{eq: if Yinfty is negative} cannot hold for all $n$. The obtained  contradiction finalizes the proof.
\end{proof}

For arbitrary positive $\varepsilon < \min_{t\in[0,T]}\left(\Psi - \varphi(t)\right)$ and $0\le t_1 < t_2 \le T$ denote
\[
\tilde {\mathcal D}^{[t_1, t_2]}_\varepsilon := \{(t,y)~|~t\in[t_1, t_2], y\in [\varphi(t)+\varepsilon, \Psi]\},
\]
where $\Psi$ is from Proposition \ref{finiteness of Y infty}, and observe that $\tilde {\mathcal D}^{[t_1, t_2]}_\varepsilon$ is a compact set and $b$ is continuous on it. Consider also 
\[
\tau_0 := \sup\{t \in [0,T]~|~\forall s\in[0,t):~Y^{(\infty)}_s > \varphi(s)\}.
\]
It is clear that $\tau_0 >0 $ because $Y^{(\infty)}$ is bounded from below by continuous processes $Y^{(n)}$ which start from the level $Y_0 > \varphi(0)$.

\begin{proposition}\label{prop: local solution}
\begin{enumerate}
\item $Y^{(\infty)}$ is continuous at any $t$ such that $Y^{(\infty)}_t > \varphi(t)$.
\item For any $t<\tau_0$:
\[
Y^{(\infty)}_t = Y_0 + \int_0^t b(s, Y^{(\infty)}_s)ds + Z_t.
\] 
\item $Y^{(\infty)}_{\tau_0} = \varphi(\tau_0)$ and, furthermore, $Y^{(\infty)}$ is left continuous at $\tau_0$:
\[
\lim_{t\to\tau_0 -} Y^{(\infty)}_t = \varphi(\tau_0).
\]
\end{enumerate}
\end{proposition}

\begin{proof}
\textbf{1.} Let $t\in[0,T]$ be such that $Y^{(\infty)}_t > \varphi(t)$. Then there exists $n_1 \ge n_0$ such that for all $n\ge n_1$: $Y^{(n)}_t > \varphi(t)$. Furthermore, because of monotonicity with respect to $n$ and continuity of both $Y^{(n)}_\cdot$ and $\varphi$, there is such $\varepsilon_1 = \varepsilon_1(n_1)$ that for any $s\in [t - \varepsilon_1, t + \varepsilon_1]$: $Y^{(n)}_s > \varphi(s)$, $n\ge n_1$. Furthermore, since for all $s\in[t-\varepsilon_1, t+\varepsilon_1]$ and $n\ge n_0$: $Y^{(n)}_s < \Psi$, for all $n\ge n_1$:
\[
\left(s,Y^{(n)}_s\right) \in \tilde{\mathcal D}^{[t-\varepsilon_1, t+\varepsilon_1]}_{\varepsilon_0},
\]
with $\varepsilon_0 := \min_{r \in [t-\varepsilon_1, t+\varepsilon_1]} \left(Y^{(n_1)}_r - \varphi(r)\right) > 0$. Therefore, if $n_2\ge n_1$ is such that 
\[
n_2 > \max_{(s,y)\in \tilde{\mathcal D}^{[t-\varepsilon_1, t+\varepsilon_1]}_{\varepsilon_0} } b(s,y),
\]
for any $n\ge n_2$ and $s\in[t-\varepsilon_1, t+\varepsilon_1]$: $b_n(s, Y^{(n)}_s) = b(s, Y^{(n)}_s) - \frac{1}{n}$, whence
\begin{equation}\label{eq: representation of Yn on interval of positivity}
\begin{aligned}
Y^{(n)}_s &= Y^{(n)}_{t-\varepsilon_1} + \int_{t-\varepsilon_1}^s b_n(r,Y^{(n)}_r)dr + Z_s - Z_{t-\varepsilon_1}
\\
&= Y^{(n)}_{t-\varepsilon_1} + \int_{t-\varepsilon_1}^s b(r,Y^{(n)}_r)dr - \frac{s - t +\varepsilon_1}{n}+ Z_s - Z_{t-\varepsilon_1}.
\end{aligned}
\end{equation}
From the choice of $n_2$, for any $n\ge n_2$ and $u \in [t-\varepsilon_1, s]$: $(u, Y^{(n)}_{u}) \in \mathcal D_{\varepsilon_0}$, therefore, by assumption \textbf{(A2)} and Proposition \ref{finiteness of Y infty}, there exists a constant $L>0$ that does not depend on $n$ such that for any $u \in [t-\varepsilon_1, s] \subset  [t-\varepsilon_1, t+\varepsilon_1]$
\[
|b(u, Y^{(n)}_{u})| \le L(1 + |Y^{(n)}_{u}|) \le L(1 + \Psi )  < \infty,
\]     
therefore, by dominated convergence,
\[
\lim_{n\to\infty} \int_{t-\varepsilon_1}^s b(r, Y^{(n)}_r)dr = \int_{t-\varepsilon_1}^s b(r, Y^{(\infty)}_r)dr
\]
which, together with \eqref{eq: representation of Yn on interval of positivity}, implies
\[
Y^{(\infty)}_s = Y^{(\infty)}_{t-\varepsilon_1} + \int_{t-\varepsilon_1}^s b(r, Y^{(\infty)}_r)dr + Z_s - Z_{t-\varepsilon_1}, \quad s \in [t-\varepsilon_1, t+\varepsilon_1].
\]
Hence $Y^{(\infty)}$ is continuous on $[t-\varepsilon_1, t+\varepsilon_1]$ and, in particular, at point $t$.
\\
\textbf{2.} Since $Y^{(\infty)}$ is greater than $\varphi$ on an arbitrary interval $[0,t]\subset[0,\tau_0)$, it is continuous on this interval. Therefore, by Dini's theorem, $Y^{(n)}$ converges uniformly to $Y^{(\infty)}$, $n\to\infty$, on $[0,t]$. Let $n_3$ be such that for all $n\ge n_3$: $Y^{(n)}_s - \varphi(s) > \frac{\min_{r\in[0,t]}(Y^{(\infty)}_r - \varphi(r))}{2} =: \varepsilon_\infty$, $s\in[0,t]$. For any $s\in[0,t]$ and $n\ge n_3$ it holds that
\[
(s,Y^{(n)}_s) \in \tilde{\mathcal D}^{[0,t]}_{\varepsilon_\infty},
\]
so, if $n_4 \ge n_3$ is such that $n_4 > \max_{(s,y) \in \tilde{\mathcal D}^{[0,t]}_{\varepsilon_\infty}} b(s,y)$,
for any $s\in[0,t]$ and $n\ge n_4$: $b_n(s, Y^{(n)}_s) = b(s, Y^{(n)}_s) - \frac{1}{n}$. Taking into account that 
\[
(s, Y^{(n)}_s),(s, Y^{(\infty)}_s) \in \mathcal D_{\varepsilon} 
\] 
for any $s\in[0,t]$ with $\varepsilon \in(0, \varepsilon_\infty)$, we have that, by assumption \textbf{(A2)}, there exists a constant $c_{\varepsilon}$ that does not depend on $n$ such that 
\[
|b_n (s,Y^{(n)}_s) - b (s, Y^{(\infty)}_s)| \le c_{\varepsilon} |Y^{(n)}_s - Y^{(\infty)}_s| + \frac{1}{n},\quad s\in[0,t],
\] 
whence $b_n (s, Y^{(n)}_s) \rightrightarrows b(s, Y^{(\infty)}_s)$ on $[0,t]$, $n\to\infty$. Now the claim can be verified by transition to the limit under the integral. 
\\
\textbf{3.} First, note that $Y^{(\infty)}_{\tau_0} = \varphi(\tau_0)$. Indeed, by Proposition \ref{prop: positivity of Yinfty}, $Y^{(\infty)}_t \ge \varphi(t)$ for all $t\in [0,T]$ and, if $Y^{(\infty)}_{\tau_0} > \varphi(\tau_0)$, then $Y^{(\infty)}$ is continuous at $\tau_0$ and therefore exceeds $\varphi$ on some interval $[\tau_0, \tau_0 + \delta)$, that contradicts the definition of $\tau_0$. Now it is sufficient to verify that 
\[
\limsup_{t\to\tau_0-} Y^{(\infty)}_t = \varphi(\tau_0).
\]
Assume it is not true and there is such $x \in (0,\infty)$ that 
\[
\limsup_{t\to\tau_0-} Y^{(\infty)}_t = \varphi(\tau_0) + x.
\]
Note also that $x < \infty$ since, by Proposition \ref{finiteness of Y infty}, $Y^{(\infty)}$ is bounded from above by the (random) constant $\Psi$.
\\
Let $\delta_x$ be such that for any $t\in[\tau_0 - \delta_x, \tau_0]$: $|\varphi(t) - \varphi(\tau_0)| < \frac{x}{4}$. Denote 
\[
\varepsilon_x := \min_{t \in [\tau_0 - \delta, \tau_0]} \left(\varphi(\tau_0) + \frac{x}{4} - \varphi(t))\right)
\]
and observe that $\varepsilon_x > 0$ and $\varphi(t) + \varepsilon_x \le \varphi(\tau_0) + \frac{x}{4}$ whenever $t \in [\tau_0 - \delta, \tau_0]$.
\\
If $x>0$, for any $\delta \in (0, \delta_x)$ there is such $t_\delta \in(\tau_0 - \delta, \tau_0)$ that $Y^{(\infty)}_{t_\delta} \ge \varphi (\tau_0) + \frac{3x}{4}$. Let such $\delta \in (0,\delta_x)$ and $t_\delta$ be fixed. Since $Y^{(n)}_{t_\delta} \uparrow Y^{(\infty)}_{t_\delta}$, $n\ge 1$, there is such $n_{\delta}$ that for all $n\ge n_{\delta}$: 
\[
Y^{(n)}_{t_\delta} \ge \varphi(\tau_0) + \frac{x}{2}.
\]
It is clear that $Y^{(n)}_{\tau_0} < \varphi(\tau_0)$ therefore, for $n\ge n_{\delta}$ one can consider the moment
\[
\theta^n := \inf\left\{t\in (t_\delta, \tau_0)~|~Y^{(n)}_t = \varphi(\tau_0) + \frac{x}{4}\right\}.
\]
From the continuity of $Y^{(n)}$, $Y^{(n)}_{\theta^n} = \varphi(\tau_0) + \frac{x}{4}$, so $ Y^{(n)}_{\theta^n} - Y^{(n)}_{t_\delta} < -\frac{x}{4}$. On the other hand, from definition of $\theta^n$ and Proposition \ref{finiteness of Y infty}, for all $t\in [t_\delta, \theta^n]$: 
\[
(t, Y^{(n)}_t) \in [t_\delta, \theta^n_x] \times \left[\varphi(\tau_0) + \frac{x}{4}, \Psi \right] \subset \tilde {\mathcal D}^{[\tau_0 - \delta_x, \tau_0]}_{\varepsilon_x}. 
\]
Let $\tilde n_\delta > n_\delta$ be such that
\[
\tilde n_\delta > \max_{(s,y) \in \tilde {\mathcal D}^{[\tau_0 - \delta_x, \tau_0]}_{\varepsilon_x}} b(s, y).
\]
For any $t\in [t_\delta, \theta_x]$ and $n\ge \tilde n_\delta$:
\[
b_n(t, Y^{(n)}_t) = b(t, Y^{(n)}_t) - \frac{1}{n}
\]
and, therefore, we obtain that
\begin{align*}
-\frac{x}{4} & > Y^{(n)}_{\theta^n} - Y^{(n)}_{t_\delta} = \int_{t_\delta}^{\theta^n} b_n(s, Y^{(n)}_s)ds + (Z_{\theta^n} - Z_{t_\delta})
\\
& = \int_{t_\delta}^{\theta^n} b(s, Y^{(n)}_s)ds - \frac{1}{n}(\theta^n - t_\delta) + (Z_{\theta^n} - Z_{t_\delta})
\\
& = \int_{t_\delta}^{\theta^n} b^+(s, Y^{(n)}_s)ds - \int_{t_\delta}^{\theta^n} b^-(s, Y^{(n)}_s)ds - \frac{1}{n}(\theta^n - t_\delta) + (Z_{\theta^n} - Z_{t_\delta})
\\
&\ge - \int_{t_\delta}^{\theta^n} b^-(s, Y^{(n)}_s)ds - \frac{1}{n}(\theta^n - t_\delta) -\Lambda(\theta^n - t_\delta)^{\alpha}
\\
&\ge - \left(\max_{(s,y)\in \tilde {\mathcal D}^{[\tau_0 - \delta_x, \tau_0]}_{\varepsilon_x} } b^-(s, y) + \frac{1}{n}\right) (\theta^n_x - t_\delta) -\Lambda(\theta^n_x - t_\delta)^{\lambda}
\\
& \ge - \left(\max_{y\in\tilde {\mathcal D}^{[\tau_0 - \delta_x, \tau_0]}_{\varepsilon_x}} b^-(s,y) + \frac{1}{n}\right) \delta -\Lambda \delta^{\lambda},
\end{align*}
i.e. for any $\delta \in (0, \delta_x)$:
\[
\delta \max_{y\in\tilde {\mathcal D}^{[\tau_0 - \delta_x, \tau_0]}_{\varepsilon_x}} b^-(s,y) + \Lambda \delta^{\lambda} \ge \frac{x}{4},
\]
which is not possible. The obtained contradiction implies that $x=0$, i.e.
\[
\limsup_{t\to\tau_0-} Y^{(\infty)}_t = \varphi(\tau_0).
\]
\end{proof}

Now, let us move to the proof of Theorem \ref{thm: existence of the local solution}.

\begin{proof}[{Proof of Theorem \ref{thm: existence of the local solution}}]
By Proposition \ref{prop: local solution}, $Y = Y^{(\infty)}$ indeed satisfies the equation of the required form. Let $\tilde Y$ satisfy the equation \eqref{volatility1} on $[0,t] \subset [0,\tilde\tau_0 \wedge \tau_0)$. Then it is continuous on $[0,t]$ and therefore $\min_{s\in[0,t]} (\tilde Y_s - \varphi(s)) > 0$. Let $\varepsilon := \min_{s\in[0,t]} (\tilde Y_s - \varphi(s))\wedge \min_{s\in[0,t]} (Y^{(\infty)}_s - \varphi(s))$ and choose $\tilde n$ such that for all $n\ge \tilde n$: $n > \max_{(s,y) \in \tilde{\mathcal D}^{[0,t]}_{\varepsilon}} b(s, y)$. Then $b(s, \tilde Y_s) = \tilde b_n (s, \tilde Y_s)$, $s\in[0,t]$, $n\ge \tilde n$, where $\tilde b_n$ is defined by \eqref{bn tilde}, whence
\[
\tilde Y_s = Y_0 + \int_0^s \tilde b_n (s, \tilde Y_u) du + Z_s,\quad s\in [0,t].
\]
However, $Y^{(\infty)}$ also satisfies the equation above and, since the latter has a unique solution, $\tilde Y_s = Y^{(\infty)}_s$ for all $s\in[0,t]$. Now it is easy to deduce that $\tau_0 = \tilde\tau_0$ and $\tilde Y_t = Y^{(\infty)}_t = Y_t$ for all $t\in[0,\tau_0]$.
\end{proof}
\end{document}